\documentclass[12pt,reqno]{amsart}

\usepackage[centertags]{amsmath}
\usepackage{amsthm,amsfonts,amsmath}
\usepackage{tikz, float}
\usepackage{mathtools}
\usetikzlibrary{matrix,arrows,positioning,calc}
\usepackage{bm}

\usepackage{amssymb}
\usepackage{epsfig}
\usepackage{amssymb, latexsym}
\usepackage{indentfirst}
\usepackage[colorlinks=true]{hyperref}
\hypersetup{allcolors=[rgb]{0.1,0.1,0.4}}
\usepackage{color}
\usepackage{verbatim}

\newcommand{\K}{{\mathbb K}}
\begin{document}
\theoremstyle{plain}
\newtheorem{theorem}{Theorem}[section]
\newtheorem{lemma}{Lemma}[section]
\newtheorem{proposition}{Proposition}[section]
\newtheorem{corollary}{Corollary}[section]
\newtheorem{example}{Example}[section]
\theoremstyle{definition}
\newtheorem{notations}[theorem]{Notations}
\newtheorem{notation}[theorem]{Notation}
\newtheorem{remark}[theorem]{Remark}
\newtheorem{question}[theorem]{Question}
\newtheorem{definition}[theorem]{Definition}
\newtheorem{condition}[theorem]{Condition}
\let\pf\proof
\let\epf\endproof
\numberwithin{equation}{section} 

\newcommand{\Keywords}[1]{\par\noindent {\small \textit{Keywords}\/}: #1}
\newcommand{\amsclass}[1]{\par\noindent {\small \textit{MSC(2020)}\/}: #1}

\title[]{Stacks in Einstein Gravity  and a stacky equivalence of 3D quantum gravity with gauge theory}

\author[Kadrİ İlker Berktav]{Kadrİ İlker Berktav}

\thanks{K. İ. Berktav is a postdoctoral researcher at the Institute of  Mathematics, University of Zurich,  Switzerland; e-mail: kadriilker.berktav@math.uzh.ch; K. İ. B. acknowledges support of T\"{U}BİTAK/2219-Program Grant. 
}

\maketitle

\begin{abstract}

	In this paper, we examine stacky structures in Einstein's theory of gravity.
	In brief, we  first give a construction of the \textit{moduli stack of solutions to (vacuum) Einstein field equations} on $n$-dimensional spacetimes, with vanishing cosmological constant. 
	Using a similar approach, we also study Einstein's gravity on families of manifolds and define another  stack encoding this situation as well. Secondly, we focus on the gauge theoretical interpretation of 3D gravity and the concept of equivalence of 3D quantum gravity with gauge theory. By \textit{equivalence}, we essentially mean the existence of an isomorphism between the phase spaces of 3D gravity and the associated gauge theory. In this regard, we show that once it exists, the equivalence induces an isomorphism between the corresponding moduli stacks.

\vspace{0.5cm}

\amsclass{14A20, 14A30, 14D23, 18F20, 70S15, 81T35, 83C99}

\end{abstract}

\tableofcontents
\section{Introduction}\label{intro on 2+1 gravity}

 Derived algebraic geometry (DAG) essentially provides a new setup to deal with non-generic situations  in geometry (e.g. non-transversal intersections and ``bad" quotients). To this end, it combines higher categorical objects and homotopy theory with many tools from homological algebra. Hence, roughly speaking, it can be viewed as \emph{a higher categorical/homotopy theoretical refinement of classical algebraic geometry}. In that respect, DAG offers  new ways of organizing information for various purposes. Therefore, it has many interactions with other mathematical domains. For a survey of some directions, we refer to \cite{Anel,Toen}.

 Regarding physics-related problems, for example, \cite{Cosv2} studies gauge theories and factorization algebras in the context of DAG. 
 In \cite{Benini},  a  stacky formulation of Yang–Mills fields on Lorentzian manifolds is presented. \cite{Benini2} examines higher structures in algebraic quantum field theory.  \cite{GradyPavlov,LudewigStoffel} focus on geometric functorial field theories.
 
 Inspired by the aforementioned formulations, in this paper, we present a similar type of analysis in the case of certain Einstein gravities, and  investigate its possible consequences.  For instance, we study 3D Einstein-Cartan-Palatini gravity theory and 3D quantum gravity using the language of stacks. Our constructions, in fact, employ some techniques from Hollander's homotopy theory of stacks \cite{Hollander}. 

The current work is centered around the fact that the phase spaces that we are mostly interested in have the structure of a \textit{groupoid}, rather than a \textit{set}. To be more specific,  for  ordinary field theories, the collection of fields have the structure of a \emph{set}, and hence two fields $f,f'$ are said to be the \textit{same} if and only if the equation $f=f'$ holds set theoretically. However, for gauge theories,  two gauge fields $A, A'$  are the ``same'' if there exists a gauge transformation $g:A\rightarrow A'$ relating them. Due to the this extra datum, points in the corresponding phase space naturally form a groupoid. I.e. the data should include the points (the fields), along with invertible (gauge) transformations between them.  Consequently, the phase space of a gauge theory turns out to be a ``higher space'' (called a \emph{stack}) rather than an ``ordinary space''. More details can be found in \cite{Benini, Benini2}.

Of course, one can naturally ask for similar kinds of relations between gauge transformations themselves. For instance, if there are gauge transformations between gauge transformations, then the underlying structure of the collection of points will be encoded by ``2-groupoids''. One can play the same game for these ``2nd level transformations'' and ends up with 3-groupoids, and so on...Using a higher categorical dictionary, this essentially leads to the notion of an \textit{infinite tower of equivalences}. Therefore, if we allow higher symmetries in  gauge theory, the natural framework will be encoded by  $\infty$-groupoids, and hence the corresponding phase space becomes a \emph{higher stack.} For more details, we again refer to \cite{Benini,Benini2}.

\newpage
It should be clear by now why it is natural to investigate similar structures in Einstein's theory of general relativity: \emph{Once symmetries are involved as a part of the data, one should interpret phase spaces as higher spaces, rather than just ordinary spaces.} This can eventually lead to a new way of formalizing the data and make certain higher algebraic tools available.  In this paper, we only consider  ``first'' level symmetries of the theory. Therefore, stacks naturally enter the picture, and they are good enough to encode the underlying structure of the phase space. In short, \emph{stacks are good enough for our purposes}, and so we concentrate on stacky constructions. 

Note that in Einstein's theory of GR, we consider pseudo-Riemannian metrics over some base space $U$ as \emph{fields}. So, the metric is the fundamental object of study encoding geometric and gravitational features of the spacetime. Moreover, regarding the symmeties in this context, the (action of) theory itself is invariant under spacetime diffeomorphisms. Therefore, for a pair $(g,g')$ of fields, for instance, one can define \emph{an invertible morphism} $\phi: g \rightarrow g'$ by using the pull-back operation, where $\phi$ can be determined by an element in the group of diffeomorphims of the underlying spacetime. More generally, viewing any such $g$ as a section of the bundle $Sym^2(T^*U)$, one may use certain automorphisms of the bundle $Sym^2(T^*U)$ and their actions on $g$ to describe  invertible morphisms between fields. Therefore, in either case, the corresponding phase space will have the structure of a \emph{groupoid}. In fact, elaborating the last statement will be one of the main objectives in this paper. 

\textbf{Main results and the outline.} Now, let us summarize our results. In this paper, using the homotopy theory of stacks \cite{Hollander}, we first give an elementary construction of the so-called \emph{moduli stack of vacuum Einstein gravities on Lorentzian spacetimes with vanishing cosmological constant}. More precisely, we prove the following result.
\begin{theorem}\label{THM 1}
	Given a  Lorentzian $ n $-manifold $M$, let $\mathcal{C}$ be the category of open subsets of  $M$ that are diffeomorphic to $\mathbb{R}^n$, 
	with morphisms being canonical inclusions between open subsets whenever $U\subset V$. Then the presheaf $\mathcal{E}\in PSh(\mathcal{C}, Grpds)$
	\begin{equation*}
	\mathcal{C}^{op} \longrightarrow Grpds, \ \ U \mapsto  \mathcal{E} (U)
	\end{equation*} is a stack of   Ricci-flat Lorentzian metrics on $\mathcal{C}$, where for an object $U$ of      $\mathcal{C}$, objects  of $ \mathcal{E}(U) $ form the set $ Ob(\mathcal{E}(U)):=\big\{ g\in\Gamma(Sym^2(T^*U)): Ric(g)=0 \big\}$, and a morphism in $ \mathcal{E}(U) $ is determined by an automorphism of $ Sym^2(T^*U) $. 
\end{theorem}Here, $Grpds$ denotes the 2-category of groupoids. Roughly speaking, $\mathcal{E}$ is a prestack \textit{(a presheaf of groupoids)} that preserves certain structures and possesses the descent property with respect to the underlying site structure on $\mathcal{C}$. 
The  precise description of  $\mathcal{E}$ as a prestack is given in Lemma \ref{main result on being pre-stack}, while  the descent property and the site structure are discussed in $\S$ \ref{stacky} (in the poof of Theorem \ref{THM 1}). 
\newpage
Theorem \ref{THM 1} provides  a suitable stack that in fact captures \emph{the contravariance and locality} behaviors of the Ricci-flat geometric structure on the underlying manifold $M$. On the other hand, in the context of moduli theory,  it  is natural to study \textit{smoothly varying families of manifolds} as well. Therefore, we also investigate Ricci-flat Lorentzian metrics on families of manifolds and define a new stack encoding this situation. 

To be more specific, we require  geometric structures to vary in families, parametrized over cartesian spaces. In brief, this can be achieved  by replacing the category $\mathcal{C}$ in Theorem \ref{THM 1} by the site $Fam_n$ of families of manifolds, where its objects are submersions $\pi: M \rightarrow S, $ with $n$-dimensional fibers, and morphisms are fiberwise open embeddings. With this modification, we prove the following result:
\begin{theorem}\label{THM 1.1}
	Let $Fam_n$ be the site of families of manifolds (with $n$-dimensional fibers). Denote an object of $Fam_n$ by $ M/S$. Then the presheaf $ \mathcal{E}^{fam}$ on $Fam_n$
	\begin{equation*}
	Fam_n^{op} \longrightarrow Grpds, \  M/S \mapsto  \mathcal{E}^{fam} (M/S),
	\end{equation*} is a  stack, where $ Ob(\mathcal{E}^{fam} (M/S)):=\{g\in \Gamma(Sym^2(T^*(M/S))) :  Ric(g)=0 \} $, and  morphisms are determined by certain automorphisms of $ Sym^2(T^*(M/S))$.  Note that $ T^*(M/S)$ denotes the relative cotangent bundle $ T^*(M/S)=Coker(T^*S\rightarrow T^*M). $
\end{theorem}
Last but not least, we revisit 3D gravity and its gauge theoretical interpretation \textit{(namely, the Einstein-Cartan-Palatini formulation)} in a particular setup. In this framework, we also examine  the so-called ``equivalence'' of 3D quantum gravity with gauge theory. Our setup in fact consists of  vacuum 3D Einstein gravity (with vanishing cosmological constant $\Lambda$) on  Lorentzian spacetimes of the form $M:=\Sigma\times \mathbb{R}$, where $\Sigma$ is a closed Riemann surface of genus $g>1$. Let us denote  this theory by $GR_{\Lambda=0}^{3D}(M)$ and  the corresponding gauge theory by $CS^{3D}_{ISO(2,1)}(M)$. 

By  \emph{equivalence}, we essentially mean the existence of an isomorphism between the phase spaces of these theories 
$$Mod(GR_{\Lambda=0}^{3D}(M)) \xrightarrow{\sim} Mod(CS^{3D}_{ISO(2,1)}(M)),$$ 
which sends a flat pseudo-Riemannian metric $g$ to the corresponding flat gauge field $A^g$.
More details will be discussed in Sections \ref{section_overview GR} and \ref{stacky equivalence}, but the upshot is that once there exists such an equivalence, one can construct a natural stack isomorphism between the stacks of these theories. Here, by a \emph{stack of a theory}, we mean the moduli stack of solutions to the corresponding field equations of the theory under consideration. In this regard, we prove the following result.
\begin{theorem}\label{THM-2}
	Suppose that $M=\Sigma\times (0,\infty)$ is a Lorentzian 3-manifold, where $\Sigma$ is a closed Riemann surface of genus $g>1$. Let $\mathcal{E}$ and $\mathcal{M}$ denote the moduli stacks of $ GR_{\Lambda=0}^{3D}(M) $ and $ CS^{3D}_{ISO(2,1)}(M)$, respectively. Then there exists an induced invertible natural transformation 
$ 	\Phi: \mathcal{E} \Rightarrow\mathcal{M}. $

\end{theorem}

Now, let us outline the remainder of this paper. $\S$\ref{background} includes  preliminaries. It begins with outlining some key ideas from Hollander's work \cite{Hollander} on the homotopy theory of stacks. In $\S$ \ref{section_overview GR}, we review 3D gravity to some extent and its gauge theoretical interpretation. 
In $\S$ \ref{prestack}, we first present an elementary construction of the moduli prestack of  Einstein gravity (cf. Lemma \ref{main result on being pre-stack}). Then we give the proof of  Theorem \ref{THM 1} 
using the homotopy theory of stacks. In $\S$ \ref{section_proof on families}, we explain the content of Theorem \ref{THM 1.1} in more detail 
and give a sketch of the proof. Finally,  $\S$ \ref{stacky equivalence} provides the proof of Theorem \ref{THM-2}.
\vspace{3pt}

\paragraph{\bf Acknowledgments} 

 It is a pleasure to thank Alberto Cattaneo, Ödül Tetik and Neeraj Deshmukh for helpful discussions and useful comments. I also thank Ödül Tetik for many discussions about geometric structures in view of stacks. 
 I would also like to thank the Institute of Mathematics, University of Zurich, where some parts of this paper were revisited and improved. The author is supported by the Scientific and Technological Research Council of Turkey (T\"{U}BİTAK) under 2219-International Postdoctoral Research Fellowship Program (2021-1).



 

\section{Background material} \label{background}

\subsection{Background from the homotopy theory of stacks}
In this section, we outline the basics of homotopy theory of stacks, present some material relevant to this paper, and state some useful results from \cite{Benini,Hollander}.

It is very well-known that by Yoneda's embedding,  one can  realize algebro-geometric objects (like schemes, stacks, derived ``spaces", etc...) as \textit{functors} in addition to the standard ringed-space formulation. 
In brief, we have the following enlightening diagram from \cite{Vezz2} encoding such a functorial interpretation: 
\begin{equation} \label{the diagram}
	\begin{tikzpicture}
	\matrix (m) [matrix of math nodes,row sep=2em,column sep=4em,minimum width=5 em] {
		CAlg_{\mathbb{K}}   & Sets  \\
		&  Grpds \\
		cdga_{\mathbb{K}}^{\leq0} &  Ssets \\};
	\path[-stealth]
	(m-1-1) edge  node [left] {{\small $ $}} (m-3-1)
	edge  node [above] {{\small schemes}} (m-1-2)
	(m-1-1) edge  node [below] {} node [below] {{\small stacks}} (m-2-2)
	(m-1-1) edge  node [below] {} node [below] {{\small $ n $-stacks}} (m-3-2)
	(m-3-1) edge  node  [below] {{\small derived stacks}} (m-3-2)
	
	(m-1-2) edge  node [right] {{\small }} (m-2-2)
	(m-2-2) edge  node [right] {{\small }} (m-3-2);
	\end{tikzpicture}
\end{equation} Here $ CAlg_{\mathbb{K}} $ denotes the category of commutative $\K$-algebras. Denote by $St_{\K}$ the $\infty$-category of (higher) $\K$-stacks, where objects in $St_{\K}$, roughly speaking, are defined via  Diagram \ref{the diagram} as certain functors, with nice geometric properties. 

Diagram \ref{the diagram} may provide further information about (higher) spaces under consideration. In the case of schemes, for instance, such a functorial description implies that the points of a scheme form a \textit{set}. Likewise, it implies that the collection of points of a stack has the structure of a \textit{groupoid}, and not that of a set. These kinds of interpretations, in fact, suggest the name ``functor of points$ " $. In brief, the right hand side of the diagram in fact encodes the structure of points.

 The RHS of  Diagram \ref{the diagram} also captures the level of symmetries and leads to the different ways of organizing the moduli data. That is, the RHS is also about \emph{how to test two objects being the same}. On the other hand, the LHS of the diagram  encodes the change in the \emph{local algebraic models} of (higher) spaces. 



In this paper, we work within the context of Hollander's theory of stacks \cite{Hollander}. In what follows,   we just  intend to give a brief sketch for the objects and constructions that we will be mostly interested in. We essentially follow \cite{Benini, Hollander}.


The punchline of the work \cite{Hollander} is that the homotopy theoretical approach essentially encodes descent properties of a stack in a rather compact way. It is in fact based on \textit{the model structure}  on the 2-category $Grpds$ of groupoids and that on the category $PSh(\mathcal{C},Grpds)$ of presheaves in groupoids. 
From \cite{Benini,Hollander}, we have the following definition/theorem, which allows us to formulate the classical notion of a Deligne-Mumford stack in the language of homotopy theory. 

\begin{definition} \label{definition of stack}
	Let $\mathcal{C} $ be a site. A \emph{stack} is a presheaf of groupoids $ \mathcal{X}:\mathcal{C}^{op}\longrightarrow Grpds $ such that for each covering family $ \{U_i\rightarrow U\} $ of $U$, the canonical morphism \begin{equation*}
	\mathcal{X}(U)\longrightarrow holim_{ {\footnotesize Grpds}} \big( \mathcal{X}(U_{\bullet}) \big)
	\end{equation*} is a weak equivalence  in $Grpds$,  where 
	\begin{equation*}
	\mathcal{X}(U_{\bullet}):= \bigg(\prod_i\mathcal{X}(U_i)\mathrel{\substack{\textstyle\rightarrow\\[-0.1ex]
			\textstyle\rightarrow \\[-0.1ex]}} \prod_{ij}\mathcal{X}(U_{ij})
	\mathrel{\substack{\textstyle\rightarrow\\[-0.1ex]
			\textstyle\rightarrow \\[-0.1ex]
			\textstyle\rightarrow\\[-0.3ex]}} \prod_{ijk}\mathcal{X}(U_{ijk})
	\mathrel{\substack{\textstyle\rightarrow\\[-0.1ex]
			\textstyle\rightarrow \\[-0.1ex]
			\textstyle\rightarrow \\[-0.1ex]
			\textstyle\rightarrow \\[-0.3ex]}} 
	\cdot\cdot\cdot \bigg)
	\end{equation*}  is the cosimplicial diagram in $Grpds$, and $U_{i_1 i_2 ...i_m}$ denotes the fibered product of $U_{i_n}$'s in $U$, that is 
	$ U_{i_1 i_2 ... i_m}:=U_{i_1}\times_U U_{i_2} \times_U \cdot \cdot \cdot \times_U U_{i_m}. $
	
\end{definition}
Definition \ref{definition of stack} essentially says that \emph{a stack} is a fibrant object $\mathcal{X}$ in the local model structure on $PSh(\mathcal{C},Grpds)$. Equivalently, \emph{a stack} can be viewed as a category fibered in groupoids over $\mathcal{C}$ satisfying some descent properties.

\begin{remark}
	\begin{enumerate}
		\item []
		\item The weak equivalences in Definition \ref{definition of stack} are those morphisms in $Grpds$ which are fully faithful and essentially surjective. Thus, for a stack $\mathcal{X}$, the canonical map above is an equivalence of categories.
		\item We will not describe  the notion of $``ho(co)lim_{Grpds} " $ in full detail. For a complete construction of this item,  we refer to \cite[$\S$ 2]{Hollander}. The following lemma, on the other hand, does provide an explicit characterization of $`` holim_{Grpds}(\cdot) " $ as a particular groupoid. Thus, we simply define the \textit{homotopy limit} $ holim_{Grpds}(X_{\bullet})$ via this characterization.
	\end{enumerate}
\end{remark} 

\begin{lemma} \cite[ Corollary 2.11]{Hollander} \label{holim as groupoid} 
	Given a cosimplicial diagram  $X_{\bullet} $ in $Grpds$ of the form 
	\begin{equation*}
	X_{\bullet}= \bigg(X_0\mathrel{\substack{\textstyle\rightarrow\\[-0.1ex]
			\textstyle\rightarrow \\[-0.1ex]}} X_1
	\mathrel{\substack{\textstyle\rightarrow\\[-0.1ex]
			\textstyle\rightarrow \\[-0.1ex]
			\textstyle\rightarrow\\[-0.3ex]}} X_2
	\mathrel{\substack{\textstyle\rightarrow\\[-0.1ex]
			\textstyle\rightarrow \\[-0.1ex]
			\textstyle\rightarrow \\[-0.1ex]
			\textstyle\rightarrow \\[-0.3ex]}} 
	\cdot\cdot\cdot \bigg), 
	\end{equation*}
	where each $X_i$ is a groupoid, then $ holim_{Grpds}(X_{\bullet}) $ is a \textit{groupoid} for which
	\begin{enumerate}
		\item [\textit{(i)}] \textit{\textbf{objects}} are the pairs $(x,h)$, where $x$ is an object in $X_0$, $h:d_1^1(x)\rightarrow d_0^1(x)$ is a morphism in $X_1$ such that 
		\begin{align}
		& (a) \ \  s_0^0(h)=id_x, \label{key lemma objects property a}\\
		& (b) \ \  d_0^2 \circ d_2^2 (h)=d_1^2(h). \label{key lemma objects property b}
		\end{align}
		Note that  $x$ and $h$ can be realized as  0- and 1-simplicies in $ X_{\bullet}, $ respectively, such that, by using the properties of $d_i^n$ and $s_j^n$, those conditions correspond to the commutativity of the diagram  
		\begin{equation*}
		\begin{tikzpicture}
		\matrix (m) [matrix of math nodes,row sep=3em,column sep=4em,minimum width=2em] {
			d_2^2 \circ d_1^1 (x) & d_2^2 \circ d_0^1 (x)=d_0^2 \circ d_1^1 (x) & d_0^2 \circ d_0^1 (x)\\
			d_1^2\circ d_1^1(x)&  & d_1^2 \circ d_0^1(x),\\};
		\path[-stealth]
		(m-1-1) edge  node [right] {{\footnotesize $ ``="$ }} (m-2-1)
		edge  node [above] {{\small $ d_2^2(h) $}} (m-1-2)
		(m-1-2) edge  node [above] {{\small $d_0^2(h)$}} (m-1-3)
		(m-1-1.east|-m-2-1) edge  node {} node [above] {{\small $ d_1^2(h) $}} (m-2-3)
		(m-1-3) edge  node [right] {{\footnotesize $``="$ }} (m-2-3);
		
		\end{tikzpicture}
		\end{equation*}
		and hence we diagrammatically have 
		\begin{equation*}
		\setlength{\unitlength}{0.9cm}
		\begin{picture}(12,4)
		\thicklines
		\put(9,3){\circle*{0.1}}
		\put(8.3,3.3){\small $d_0^2 \circ d_0^1 (x)$}
		\put(8,1){\circle*{0.1}}
		\put(6.3,0.4){\small $d_2^2 \circ d_1^1 (x)$}
		\put(10,1){\circle*{0.1}}
		\put(10.2,0.5){\small $d_2^2 \circ d_0^1 (x)$}
		\put(8,1){\line(1,0){2}}
		\put(8,1){\line(1,2){1}}
		\put(10,1){\line(-1,2){1}}
		\put(8.7,0.5){{\small $ d_2^2(h) $}}
		\put(7.2,1.9){{\small$ d_1^2(h) $}}
		\put(9.8,1.9){{\small $d_0^2(h) $}}
		
		\put(1,2.2){$\bullet$}
		\put (0.6,2){$ x $}
		\qbezier(1.4,2.7)(2,3.5)(2.8,3.1)
		\put(2.8,3.1){\vector(1,0){0.1}}
		\qbezier(1.5,2)(2,1.1)(4.8,1.5)
		\put(4.8,1.5){\vector(1,1){0.1}}

		\put(3.6,2.3){\vector(-1,0){2}}
		\put(2.8,1.8){\small $  s_0^0 $}
		
		\put(3.1,3){\line(3,-2){2}}
		\put(4.1,2.5){\small$ h $}
		\put(3.1,3.2){\small$d_1^1(x)$}
		\put(3.1,3){\circle*{0.1}}
		\put(5.2,1.90){\small $d_0^1(x)$}
		\put(5.1,1.66){\circle*{0.1}}
		\qbezier(4.8,2.8)(6,4.5)(8,2.5)
		\put(8,2.5){\vector(1,-1){0.1}}
		\put(5.5,3.8){ \small $ d_1^2 $}
		
		\end{picture}
		\end{equation*}
		\item [\textit{(ii)}] \textit{\textbf{morphisms }}are the arrows of pairs $(x,h) \rightarrow (x',h')$ that consist of a morphism $ f:x\rightarrow x' $ in $X_0$ such that the following diagram commutes.
		\begin{equation*}
		\begin{tikzpicture}
		\matrix (m) [matrix of math nodes,row sep=3em,column sep=4em,minimum width=2em] {
			d_1^1(x) & d_1^1(x')\\
			d_0^1(x) & d_0^1(x')\\};
		\path[-stealth]
		(m-1-1) edge  node [left] {{\small $ h$}} (m-2-1)
		edge  node [above] {{\small $ d_1^1(f) $}} (m-1-2)
		(m-2-1.east|-m-2-2) edge  node [below] {} node [below] {{\small $ d_0^1(f) $}} (m-2-2)
		(m-1-2) edge  node [right] {{\small $h'$}} (m-2-2);
		\end{tikzpicture}
		\end{equation*}
		Here, $ d_i^n $'s are in fact covariant functors between groupoids.
	\end{enumerate}
\end{lemma}

\subsection{Overview of 3D gravity, Cartan's formalism, and 3D quantum gravity} \label{section_overview GR}
In this section, we briefly discuss 3D Einstein gravity, infinitesimal symmetries, the Cartan formalism, and some aspects of 3D quantum gravity (and its relation with gauge theory). As each subject itself is quite dense, we can only present some key ideas and results from the literature that are relevant to our goals. But, we  try to provide a list of accessible references on each subject. 

\subsubsection{\textbf{Basics of 3D gravity}} In GR, the metric tensor is the fundamental field of study. In the context of the usual metric formalism in three dimensions,  we consider \emph{the standard  Einstein-Hilbert action} for the metric \begin{equation}\label{EH action}
\mathcal{I}_{EH}[g]:= \kappa\displaystyle \int_M \mathrm{dx^3}R \sqrt{-det(g)},
\end{equation}where $\kappa$ is a constant, $R$ is the Ricci scalar, $g$ is the metric tensor field, and $ det(g) $ denotes the determinant of the metric tensor matrix. Then, the vacuum Einstein field equations, with cosmological constant $\Lambda=0$, are given as 
\begin{equation} \label{eqn_vacumm Einstein equation}
R_{\mu \nu}-\frac{1}{2}g_{\mu \nu}R=0.
\end{equation}Observe that after contracting with $g^{\mu \nu}$, one has $R=0$. Therefore, it follows directly from substituting this back into Equation (\ref{eqn_vacumm Einstein equation}) that the moduli space $\mathcal{EL} $ of solutions to those field equations turns out to be the \textit{moduli space $\mathcal{E}(M)$ of Ricci-flat ($R_{\mu \nu}=0$) Lorentzian metrics on $M$}. In other words, it is just the moduli  of \textit{flat geometric structures on $M$}. With this interpretation in hand, one can equivalently say that Lorentzian spacetime is locally modeled on ($ ISO(2,1), \mathbb{R}^{2+1} $), where  $ \mathbb{R}^{2+1} $ denotes the usual Minkowski spacetime  \cite{Carlip, Mess}. 

It should also be noted that,   Weyl tensor in 3D is identically zero. Then the Riemann tensor can locally be expressed in terms of $R$ and $ R_{\mu \nu} $,  and so we locally have $R_{\mu \nu \sigma \rho}=0$ as well. That is, \textit{any solution of the vacuum  Einstein field equations in 3D, with vanishing cosmological constant,  is locally flat}.

\subsubsection{\textbf{Symmetries in the context of Lagrangian formalism}} 
The Hamilton's action principle  allows us to study identities and conserved quantities from the symmetries of the corresponding Lagrangian, and hence invariance properties of the action under certain transformations. This approach applies not only to the trajectories of individual particles in classical mechanics, but also works for continuous fields like $g_{\mu\nu}$. 

In Newtonian mechanics, there is a translation-invariance, which leads to a conserved momentum. In GR, on the other hand, the Einstein-Hilbert action is diffeomorphism-invariant, which essentially leads to the contracted Bianchi identity \begin{equation*}
\nabla_{\mu} G^{\mu \nu}=0, \text{ where } G^{\mu \nu}= R^{\mu \nu}-\frac{1}{2}g^{\mu \nu}R.
\end{equation*}

Let us examine the types of transformations we are considering in the above cases: For the trajectories of particles $ \textbf{q}=\{q_i(t)\} $, with the action $$\displaystyle \int_I L(t,\textbf{q}, \dot{\textbf{q}})\mathrm{d}t,$$ we consider an infinitesimal symmetry operation $q_i(t) \rightarrow q_i(t) + \delta q_i(t)$. Here, the components $ \delta q_i(t) $ for the \textit{variation} $\delta \textbf{q}$ of the trajectory $\textbf{q}(t)$ can be described by a vector field $ \xi=\xi^{\mu}\partial_{\mu}$, which controls the deformation of the original trajectory. One can verify that if $ q_i(t) $ satisfies the corresponding Euler-Lagrange equation, so does $ q_i(t) + \delta q_i(t)$. For more details, we refer to \cite{Bertschinger}.

The same idea can apply to  variations of the continuous fields. For the case of Einstein-Hilbert action, we consider its change under transformations of the form 
\begin{equation}
g_{\mu \nu} (x) \rightarrow g_{\mu \nu} (x)+\delta g_{\mu \nu} (x).
\end{equation}  The Lagrangian in this case is chosen so that the action $ \mathcal{I}_{EH}[g] $ is invariant under the transformation above \emph{for the metrics satisfying Einstein field equations.}

It should be noted that the variations above are not necessarily generated by diffeomorphisms. However, to capture the diffeomorphis-invariant nature of GR, we consider certain types of variations induced by \textit{infinitesimally generated diffeomorphisms}, by which we mean diffeomorphisms that are generated by a vector field $X$. In that case, we call $X$ the \textit{infinitesimal generator} of the corresponding transformation.

\begin{remark}\label{Remark_a variation induced from an infinitesimal diffeomorphism}
	Recall that any vector field defines a one-parameter group of diffeomorphisms via its local flow. Using an infinitesimal diffeomorphism $ \phi^X $(and hence the corresponding flow), one can examine how the metric tensor field $g_{\mu \nu}$ changes when it is pulled back along the integral curves of $X$. Notice that this is exactly what the Lie derivative $\mathcal{L}_Xg_{\mu \nu}$ measures! Therefore, we introduce the following definition. \begin{definition}
		By a\textit{ variation induced from an infinitesimal diffeomorphism $ \phi^X $},  we actually mean 
	\begin{equation}
	\delta g_{\mu \nu}:=\mathcal{L}_X g_{\mu \nu},
	\end{equation} with the transformation $  g_{\mu \nu} (x) \rightarrow g_{\mu \nu} (x)+\mathcal{L}_X g_{\mu \nu} (x).$
	\end{definition}
\end{remark}
\subsubsection{\textbf{Cartan's formalism and gauge theoretic interpretation of 3D gravity}}\label{cartan formalism}

In this section, we  outline Cartan's formalism. For more details, we refer to \cite{Carlip,KKR, Wit1}. In a nutshell, Cartan's formalism consists of the following data: 

\begin{enumerate}
	\item  A section $e^a_i$ of the orthonormal frame bundle $LM$ over $M$ for each $i$. That is, \begin{equation*}
	e^a_i \in \Gamma(M,LM),
	\end{equation*} where $i$ labels the \textit{space indices} with respect to the local  chart $\big(U_i, x\big)$ around a point $p \in M$, and $a$'s are called \textit{Lorentz indices} labeling vectors in the orthonormal basis $\{e^1_i,e^2_i,...,e^{dimM^2}_i\}$ over $U_i.$ Here, each fibre  of $LM$ \begin{equation*}
	LM_p=\big\{\big(e^1_i(p),...,e^m_i(p)\big) : e^1_i(p),...,e^m_i(p) \ \text{forms a  basis  for} \ TM_p \big\}
	\end{equation*} 
	is isomorphic to $GL(n,\mathbb{R})$. Such $ e^a_i $ are called the \textit{vierbein.}
	\item A $SO(2,1)$-connection (or the \textit{spin connection}) one-form $\omega^a_{i \ b}$ on $M$. That is,  \begin{equation*}
	\omega^a_{i \ b} \in \Omega^1(M)\otimes\mathfrak{so(2,1)},
	\end{equation*} where $\omega_i$ is a Lie algebra-valued connection 1-form on $LM$ such that $\omega^a_i:=(e^a_i)^*\omega_i$.
	\item Compatibility conditions on the metric:
	\begin{equation*}
	g_{i j}=e^a_ie^b_j\eta_{ab} \ \ \text{ and } \ \ g^{ij}e^a_ie^b_j=\eta^{ab},
	\end{equation*} where $\eta$ denotes the usual Minkowski metric.
\end{enumerate}
The key is the following observation: In 3D  gravity, the vierbein and the spin connection can be considered as a pair $(e^a_i,  \omega^a_i) $ such that they could be combined into a certain gauge field $A$, with the gauge group $ISO(2,1)$. In brief, $ \omega^a_i $ in fact plays the role of the so-called $SO(2,1)$-part of the connection $A$ (the \textit{Lorentz-part}), while $e^a_i  $ corresponds to \textit{translation generators} of the Lie algebra $iso(2,1)$ of $ISO(2,1)$. For some technical reasons, the vierbein is supposed to be invertible \cite{Wit1}. Non-invertible ones can be important in the quantum theory \cite{Bertschinger,KKR}.

Employing Cartan's formalism, the usual Einstein-Hilbert action in Equation (\ref{EH action}) can be re-expressed as 
\begin{equation}\label{EHC action}
\mathcal{I}_{EH}'[e,\omega]= \displaystyle \int \limits_{M} e^a \wedge \bigg(\mathrm{d}\omega_a + \frac{1}{2}\epsilon_{abc} \omega^b\wedge\omega^c\bigg),
\end{equation} where $e^a=e^a_i\mathrm{d}x^i$ and $\omega^a=\frac{1}{2}\epsilon^{abc}\omega_{ibc}\mathrm{d}x^i$, together with an invariant non-degenerate, bilinear form $ \langle \cdot, \cdot \rangle $ on the Lie algebra $iso(2,1)$. More precisely, $ \langle \cdot, \cdot \rangle $ is defined via \begin{equation*}
\langle J_a, P_b \rangle=\delta_{ab} \ \ \langle J_a, J_b \rangle = \langle P_a, P_b \rangle =0,
\end{equation*} with the structure relations 
\begin{equation*}
[J_a,J_b]=\epsilon_{abc}J^c, \ \ [J^a,P^b]=\epsilon_{abc}P^c, \ \ [P^a, P^b]=0.
\end{equation*}

\noindent Define the gauge field $A^g$ as \begin{equation} \label{defining the gauge field A  via e,w}
A_i:=P_ae^a_i + J_a\omega^a_i,
\end{equation} where $A^g=A_i(x)dx^i$ in a local coordinate chart $x=(x^i)$ such that $P_a$ and $J_a$ correspond to translations and Lorentz generators, respectively.

 We then define the Chern-Simons theory, with gauge group $G=ISO(2,1)$, in accordance with the  bilinear form  $ \langle \cdot, \cdot \rangle $ and the gauge field $A_i$ above. Then, by using Equation (\ref{defining the gauge field A  via e,w}), the usual Chern-Simons action \begin{equation}\label{CS action}
CS[A]= \displaystyle \int\limits_{M}^{}\langle A ,  \mathrm{d}A + \frac{2}{3} A \wedge A\rangle 
\end{equation}
becomes exactly the same expression given in Equation (\ref{EHC action}). For computational details, see \cite{Carlip,KKR, Wit1}. Note that obtaining the same action functional is just one part of the whole story. We also need to verify that the diffeomorphism invariance of 3D gravity must also be encoded in some way in the $(e,\omega)$-formalism. 

As stressed explicitly in \cite{Carlip,KKR, Wit1}, the notions of invariance in these two  formalisms, i.e. the $ 2nd $-order (metric) formalism and the $1st$-order $ (e,\omega) $- formalism, are related to each other in some  sense. In fact, \textit{the invariance under spacetimes diffeomorphisms in the metric formalism corresponds to the invariance under the corresponding gauge transformations in  the $ (e,\omega) $- formalism. } 

Note that spacetime diffeomorphisms do not correspond to  independent gauge symmetries. They are indeed combinations of local Lorentz transformations and local translations \cite{KKR}.  Due to the rather expository nature of this section, we cross our fingers and avoid the derivation of these relations to save some space and time!  For a systematic treatment, we again refer to \cite{Carlip,KKR, Wit1}

\begin{remark} \label{remark_the equivalence between diffeomorphisms and  gauge transformations}
	Assuming the invertibility of vierbein, it should  be noted that the equivalence between diffeomorphisms and  gauge transformations is valid only for \textit{infinitesimally generated diffeomorphisms} and \textit{infinitesimal gauge transformations}. 
	
	Technically speaking, it has been shown by Witten \cite{Wit1} that diffeomorphisms in the connected component of the identity are equivalent to transformations combining local Lorentz transformations and local translations mentioned above. In other words, when we identify the phase space of 3D Einstein's theory with that of the associated 3D CS theory, \emph{infinitesimal CS gauge transformations are equivalent to infinitesimal diffeomorphisms}. This does not hold for ``large" diffeomorphisms, i.e. those are not infinitesimally generated. Large diffeomorphisms in fact require different treatment, and they are important for the quantum theory \cite{KKR}. Therefore, when we discuss an equivalence between some transformations, we always consider them ``infinitesimally generated".
\end{remark}Now, employing Cartan's formalism, one can reinterpret 3D gravity in the language of gauge theory (the \textit{Einstein-Cartan-Palatini gravity}). Assuming the special case, where $M=\Sigma\times \mathbb{R}$ and $\Sigma$ is a closed Riemann surface of genus $g>1$, the study of 3D gravity in fact boils down to that of $ISO(2,1)$ Chern Simons theory on $ M,$  with the action functional $CS$ given in Equation (\ref{CS action}), and the gauge group $ \mathcal{G} $ locally of the form $Map(U,ISO(2,1))$ that acts on the space $\mathcal{A}$ of $ISO(2,1)$-connections on $\Sigma$ in a natural way: For all $\rho\in \mathcal{G}$ and $A \in \mathcal{A}$, we set \begin{equation*}
A \bullet \rho := \rho^{-1}\cdot A \cdot \rho + \rho^{-1}\cdot \mathrm{d} \rho.
\end{equation*}
The corresponding E-L equation in this case turns out to be 
\begin{equation*}
F_{A}=0,
\end{equation*}
where $F_{A}=\mathrm{d} A+A \wedge A$ is the \textit{curvature two-form} on $M$ associated to $A.$ 

\subsubsection{\textbf{Equivalence of 3D quantum gravity with gauge theory}}
Using the gauge theoretic interpretation,   the physical phase space of 3D  gravity on $M=\Sigma\times (0,\infty)$ (with $ \Lambda=0 $) can be now realized as the moduli space $ \mathcal{M}_{flat} $ of flat $ISO(2,1)$-connections on $ \Sigma $. Then there is a natural map 
\begin{equation} \label{map phi}
\phi:\mathcal{E}(M)\longrightarrow \mathcal{M}_{flat}
\end{equation}
sending a flat pseudo-Riemannian metric $g$ to the corresponding flat gauge field $A^g$. 
It should be noted that $\phi$ \emph{need not to be invertible} in the first place. 

In quantum gravity, one seeks for the construction of  a quantum Hilbert space by quantizing the moduli space  $ \mathcal{E}(M) $ of solutions to the vacuum Einstein field equations on $M$. In the gauge theoretic formulation, on the other hand, one can actually quantize the phase space $\mathcal{M}_{flat}$ of the Chern-Simons theory associated to 3D gravity using the so-called \textit{geometric quantization formalism}. Thus, to construct  quantum theory of gravity, a possible strategy one may consider is as follows: \emph{First, we translate everything into a gauge theoretical framework, and view everything as gauge theory. Then, one may try to ``quantize" the corresponding gauge theory.}

 When $\Lambda=0$, as discussed above, the 3D gravity corresponds to the Chern-Simons theory with gauge group $G=ISO(2,1)$. 
Then we can discuss the quantized theories.
However, we end up with the following question: \emph{Are the resulting theories equivalent (in some sense)? } This leads to
	
\begin{definition}\label{equvalence}
	We say that  quantum gravity is \emph{equivalent} to gauge theory in the sense of the canonical formalism if the map $\phi$ in (\ref{map phi})  is an isomorphism. 
	
\end{definition}

\begin{remark}
	As noted in \cite[$\S$ 6]{Wit2}, the map $\phi$ in  (\ref{map phi}) happens to be invertible if every flat connection in  $\mathcal{M}_{flat}$ can be transformed into a form (uniquely up to  a diffeomorphism/local Lorentz transformation) in which the vierbein is invertible. In this regard, one has the following important result from \cite{Mess}, which is central for us.
\end{remark}
\begin{theorem} \label{thm_Mess reult summary}
	  For  vacuum Einstein gravity on $M=\Sigma\times (0,\infty)$, with $\Lambda=0$, and $\Sigma$  a closed Riemann surface of genus $g>1$, there exists an equivalence of quantum gravity with  gauge theory in the sense of Definition \ref{equvalence}.
	
\end{theorem}

\section{Proofs of the main results}
In what follows, we give some preliminary results, more explanations about the contents of  Theorems \ref{THM 1}, \ref{THM 1.1} \& \ref{THM-2},  and the proofs of these  results. 
\subsection{Proof of Theorem \ref{THM 1}} \label{prestack}
In this section, we will present the proof of Theorem \ref{THM 1}. Inspired by \cite{Benini}, we first prove the following result encoding the pre-stacky part of the construction for Einstein gravities.

\begin{lemma}
	\label{main result on being pre-stack}Given a  Lorentzian $ n $-manifold $M$, let $\mathcal{C}$ be the category of open subsets of  $M$ that are diffeomorphic to $\mathbb{R}^n$, 
	with morphisms being canonical inclusions between open subsets whenever $U\subset V$. Then the  functor $\mathcal{E}:\mathcal{C}^{op} \rightarrow Grpds $ described below 	
	  is a \textit{prestack}.
\end{lemma}	
	\begin{enumerate}
		\item \textbf{The action of $\mathcal{E}$ on the objects of $\mathcal{C}.$} For each object $U$ of $\mathcal{C}$, we have a groupoid $ \mathcal{E}(U)$ of Ricci-flat pseudo-Riemannian metrics on $U$,  where \emph{objects of $ \mathcal{E}(U) $} form the \textit{set} 
		\begin{equation*}
		FMet(U):= \big\{ g\in\Gamma(Sym^2(T^*U)): Ric(g)=0 \big\}.
		\end{equation*} 
		
		 \emph{Morphisms in $ \mathcal{E}(U)$}. Let $ Aut(Sym^2(T^*U)) $ be the group of automorphisms of the bundle $Sym^2(T^*U)$ over $U$, and $ \cdot \varphi$ denotes the action of $\varphi$ on the sections. We may sometimes use $\varphi^*$ for the action as well because of the natural motivation coming from the pulling-back operation. 
		
	By the \emph{action of $\varphi$}, we mean that 
	$\varphi$ is a bundle isomorphism making the  diagram 
	\begin{equation*}
	\begin{tikzpicture}
	\matrix (m) [matrix of math nodes,row sep=3em,column sep=2em,minimum width=2em] {
		Sym^2(T^*U)  & & Sym^2(T^*U) \\
		\ & U & \\};
	\path[-stealth]
	(m-1-1)	edge  node [above] {$ \varphi $} (m-1-3)
	(m-1-3) edge node [left] {$\pi$} (m-2-2)
	(m-1-1) edge node [right] {$\pi$} (m-2-2)
	(m-2-2) edge [bend left=40] node [right] {$g$} (m-1-1)
	(m-2-2) edge [bend right=40] node [left] {$g'$} (m-1-3);
	\end{tikzpicture}
	\end{equation*}	commute such that it acts on each fiber isomorphically; that is, for each $p \in U$ there is an isomorphism $\varphi_p: Sym^2(T^*_pU) \xrightarrow [] {\sim} Sym^2(T^*_pU)$ such that 
	\begin{equation}
	g'_p= \varphi_p(g_p).
	\end{equation} 
	
	In the context of GR, we consider particular automorphisms that are induced from \emph{infinitesimal diffeomorphisms} of the underlying spacetime. Following Remarks \ref{Remark_a variation induced from an infinitesimal diffeomorphism} and \ref{remark_the equivalence between diffeomorphisms and  gauge transformations}, we consider the infinitesimal diffeomorphisms acting on the metric $g$ as 
	\begin{equation*}
	g_{\mu \nu} (p) \longrightarrow g_{\mu \nu} (p)+\mathcal{L}_X g_{\mu \nu} (p),
	\end{equation*} where $ X \in \Gamma(TU) $ is a vector field over $U$,  $p\in U$, and $\mathcal{L}_X$ is the Lie derivative operator along $X$. Here,  $\mathcal{L}_X g$ serves as a variation $\delta g$ of $g$ as in Remark \ref{Remark_a variation induced from an infinitesimal diffeomorphism}. 
	
	Since any combinations of infinitesimal diffeomorphisms are also meaningful for our construction, considering the $C^{\infty}$-module generated by these infinitesimal generators over $U$, we formally define
	\begin{equation}
	L(U)= \big \langle \mathcal{L}_X : [\mathcal{L}_X, \mathcal{L}_Y] = \mathcal{L}_{[X,Y]}, \  X,Y \in \Gamma(TU) \big \rangle
	\end{equation} as an algebra over $C^{\infty}(U).$ Then we also have the following definition.
	\begin{definition}
		Let $g \in \mathcal{E}(U)$. By an \emph{infinitesimal diffeomorphism} $\varphi$, we mean a transformation determined by an element $\hat{\varphi} \in L(U)$ such that for each $p \in U$, $g$ transforms under this infinitesimal diffeomorphism as
	\begin{equation}\label{defn_fiberwise action}
g_{\mu \nu} (p) \xrightarrow {\varphi}   g_{\mu \nu} (p)+\hat{\varphi} (g_{\mu \nu}) (p).
	\end{equation} In this case, we also use $\cdot \varphi$ to denote the action of this infinitesimal transformation on the space of metrics. As mentioned before, if $ g_{\mu \nu} $ satisfies the corresponding Einstein field equations, so does its variation $ g_{\mu \nu}\cdot \varphi.$
	\end{definition}
	
	 \begin{definition}
	 	We  define a \textit{morphism} $ g\rightarrow g'$ in $\mathcal{E}(U)$	if there exists an infinitesimal diffeomorphism $\varphi$ such that $ g'=  g \cdot \varphi$. Then the set of morphisms is given by
	 	\begin{equation*}
	 	Hom_{\mathcal{E}(U)}(g,g')= \big \{ \varphi \in Aut(Sym^2(T^*U)): g'=  g \cdot \varphi   \text{ in } \mathcal{E}(U) \big \}.
	 	\end{equation*}
	 \end{definition}	
		We  denote a  morphism $ g\rightarrow g'$ in $ Hom_{\mathcal{E}(U)}(g,g')$ by $ (g,\varphi)$ or just by $\varphi$ if the meaning is clear from the context. It is also clear from the construction that all morphisms in $ Hom_{\mathcal{E}(U)}(g,g')$ are invertible.	
	
		\emph{Compositions in $\mathcal{E}(U)$.} Given two morphisms $ g\xrightarrow{\psi} g'$ and $ g'\xrightarrow{\varphi} g''$in $\mathcal{E}(U)$, using Equation (\ref{defn_fiberwise action}), the composition of two morphisms is given as  the standard composition
		\begin{equation*}
		(g \cdot \psi)\cdot \varphi : g \rightarrow g'',
		\end{equation*} where  $\hat{\varphi}, \hat{\psi} \in L(U)$ representing the corresponding operators. More precisely, w.l.o.g, we assume $ \hat{\varphi}\equiv\mathcal{L}_X $ and $ \hat{\psi}\equiv\mathcal{L}_Y $ for some vector fields $X, Y$ on $U$. Then one obtains
		\begin{align*}
		g''_{\mu \nu} (p)&= g'_{\mu \nu} (p)\cdot \varphi\\
		&= g'_{\mu \nu} (p)+ \mathcal{L}_X g'_{\mu \nu} (p) \\
		&= g_{\mu \nu} (p)+\mathcal{L}_Y g_{\mu \nu} (p) + \mathcal{L}_X\big(g_{\mu \nu} (p)+\mathcal{L}_Y g_{\mu \nu} (p)\big)\\
		&= g_{\mu \nu} (p) + \big(\mathcal{L}_Y+\mathcal{L}_X + \mathcal{L}_X\mathcal{L}_Y \big)g_{\mu \nu} (p),
		\end{align*} where $ \big(\mathcal{L}_Y+\mathcal{L}_X + \mathcal{L}_X\mathcal{L}_Y \big) \in L(U)$, and we get a morphism $ g \rightarrow g''$ represented by the element $ \psi + \varphi + (\varphi \circ \psi)$. Following our notation, we use $``\varphi\circ \psi"$ to represent the composition, by which we mean $g \cdot (\varphi\circ \psi)= (g \cdot \psi)\cdot \varphi. $

		\item \textbf{The action of $\mathcal{E}$ on the morphisms in $\mathcal{C}.$} To each morphism $ U \xrightarrow{f} V $ in $\mathcal{C}$, it assigns 
	a functor of categories $ \mathcal{E}(f):\mathcal{E}(V)\rightarrow\mathcal{E}(U),$ whose action on both objects and morphisms of $\mathcal{E}(V)$ is given as follows.
		\begin{enumerate}
			\item For any object $g\in Ob(\mathcal{E}(V))=FMet(V)$, we set $ g \xrightarrow{\mathcal{E}(f)} f^*g, $ where
			\begin{equation*}
		f^*g = g \circ f= g|_U\in FMet(U). 
			\end{equation*} Notice that the pullback of a Ricci-flat metric, in general, may no longer be Ricci-flat. But, in the case of particular canonical inclusions $f: U\hookrightarrow V$, with $U,V$ open subsets, if a metric $g$ is Ricci-flat on $V$, so is $f^*g$ on $U$. This is because $f^*g$ is just the restriction $g|_U$ of $g$ to the open subset $U$. 
			\item For any morphism $ (g, \varphi) \in Hom_{\mathcal{E}(V)}(g,g')$
			,  by the definition of $\varphi$,  there is an isomorphism  
			$ g_{\mu \nu} (p) \rightarrow g'_{\mu \nu} (p)=g_{\mu \nu} (p)+\hat{\varphi} (g_{\mu \nu}) (p)$ for all $p\in U \subset V$ as well. Therefore, due to the fiberwise action given in Equation (\ref{defn_fiberwise action}), $\varphi$ induces an isomorphism  $\varphi_p: Sym^2(T^*_pU) \xrightarrow [] {\sim} Sym^2(T^*_pU)$, and hence a subbundle isomorphism. Thus, we get 
			the desired transformation over the smaller open subset $U \text{ in } V$. We denote this induced isomorphism by $ \varphi|_U $ (or  $f^*\varphi$), and write   
			\begin{equation*}
			\bigg(g\xrightarrow[ (g,\varphi) ]{\sim}  g'\bigg) \xrightarrow{\mathcal{E}(f)} \bigg( g|_U \xrightarrow [(g|_U, \varphi|_U)]{\sim} g'|_U \bigg).
			\end{equation*} 
			
		\end{enumerate}
		\item Given a composition of morphisms $ U \xrightarrow{f} V \xrightarrow{h} W$ in $\mathcal{C}$, there exists an invertible natural transformation (arising naturally from properties of the action) \begin{equation*}
		\phi_{h \circ f}:\mathcal{E}(h \circ f) \Rightarrow \mathcal{E}(f) \circ \mathcal{E}(h),
		\end{equation*} together with the compatibility condition.
	\end{enumerate}

\paragraph{\textbf{Proof of Lemma \ref{main result on being pre-stack}}} 
It is enough to prove the following two statements: \begin{itemize}
	\item[\textit{(i)}] Given a composition of morphisms 
in $\mathcal{C}$
	
	\begin{equation*}
	{\footnotesize\begin{tikzpicture}[scale=3]
		\node  (a)  {$U$};
		\node  (b) [right=of a] {$V$};
		\node  (c) [right=of b] {$W$,};
		\node (midpoint) at ($(a)!.5!(c)$) {};
		\node (psi) [above=.4cm of midpoint] {$h \circ f$};
		\draw[->]  (a) -- node [below] {$ f $} (b);
		\draw[->] (b) -- node [below] {$ h $} (c);
		\draw[->] (a.north east) parabola bend (psi.south) (c.north west);
		\end{tikzpicture}
	}
	\end{equation*}
	
	there is an invertible natural transformation 
	\begin{equation*}
	{\small \begin{tikzpicture}
		\matrix[matrix of nodes,column sep=1.6cm] (cd)
		{
			$  \mathcal{E}(W) $ & $ \ \ \mathcal{E}(U). $ \\
		};
		\draw[->] (cd-1-1) to[bend left=50] node[label=above:$  \mathcal{E}(h \circ f) $] (U) {} (cd-1-2);
		\draw[->] (cd-1-1) to[bend right=50,name=D] node[label=below:$\mathcal{E}(f) \circ \mathcal{E}(h)$] (V) {} (cd-1-2);
		\draw[double,double equal sign distance,-implies,shorten >=10pt,shorten <=10pt] 
		(U) -- node[label=right: $\psi_{h,f}$] {} (V);
		\end{tikzpicture}}
	\end{equation*}
	\item[\textit{(ii)}] Given a composition  of morphisms 
	$U\xrightarrow{f} V \xrightarrow{h} W \xrightarrow{p} Z$ in $\mathcal{C}$, the associativity condition holds in the sense that the following diagram commutes:
	\begin{equation*}
	\begin{tikzpicture}
	\matrix (m) [matrix of math nodes,row sep=3em,column sep=4em,minimum width=2em] {
		\mathcal{E}(p \circ h \circ f)  & \mathcal{E}(h \circ f) \circ \mathcal{E}(p) \\
		\mathcal{E}(f) \circ	\mathcal{E}(p \circ h) & \mathcal{E}(f)\circ \mathcal{E}(h)\circ \mathcal{E}(p)\\};
	\path[-stealth]
	(m-1-1) edge [double] node [left] {{\small $ \psi_{p\circ h,f} $}} (m-2-1)
	edge [double] node [below] {{\small $ \psi_{p, h \circ f} $}} (m-1-2)
	(m-2-1.east|-m-2-2) edge [double] node [below] {} node [below] {{\footnotesize $id_{\mathcal{E}(f)}\star\psi_{p,h} $}} (m-2-2)
	(m-1-2) edge [double] node [right] {{\small $\psi_{h,f}\star id_{\mathcal{E}(p)}$}} (m-2-2);
	\end{tikzpicture}
	\end{equation*} 
\end{itemize}

\noindent \textit{Proof of \textit{(i)}.} First, we need to analyze \emph{objectwise}: For any object $g \in FMet(W)$, we have the following \textit{strong condition} by which the rest of the proof will become rather straightforward.
\begin{equation}\label{strongcondition}
\mathcal{E}(h \circ f)(g)= (h \circ f)^*g=f^*h^*g=\big( \mathcal{E}(f) \circ \mathcal{E}(h) \big)(g)\in FMet(U).
\end{equation}

\noindent As we have identical metrics $ \mathcal{E}(h \circ f)(g) = \mathcal{E}(f) \circ \mathcal{E}(h)(g)$ for any $g \in FMet(W)$, there is, by construction, 
a unique identity morphism \begin{equation*}
\big(\mathcal{E}(h \circ f)(g),id\big) \in Hom_{FMet(U)}\Big(\mathcal{E}(h \circ f)(g), \ \mathcal{E}(f) \circ \mathcal{E}(h)(g)\Big )
\end{equation*} such that \begin{equation*}
\mathcal{E}(h \circ f)(g) \xrightarrow[\big(\mathcal{E}(h \circ f)(g),id\big)]{\sim} \mathcal{E}(f) \circ \mathcal{E}(h)(g)= \big(\mathcal{E}(h \circ f)(g)\big)\cdot id =\mathcal{E}(h \circ f)(g).
\end{equation*}
Thus, one has the natural choice of a collection of morphisms $$\big \{m_g : \mathcal{E}(h \circ f)(g)\longrightarrow \mathcal{E}(f) \circ \mathcal{E}(h)(g)\big \},$$ where
$ m_g= \big(\mathcal{E}(h \circ f)(g),id\big) \text{ for all } g \in FMet(W)$. 
\vspace{5pt}

Just for the sake of notational simplicity, we  let \begin{equation*}
\mathcal{F}:=\mathcal{E}(h \circ f) \ \ and \ \ \mathcal{G}:=\mathcal{E}(f) \circ \mathcal{E}(h).
\end{equation*} Then for each morphism $g \xrightarrow[(g,\phi)]{\sim}g'$ in $\mathcal{E}(W)$, 
we get
\begin{align*} 
\mathcal{F}((g,\phi))&= \mathcal{E}(h \circ f)((g,\phi))\nonumber \\
&=\big((h \circ f)^*g,(h \circ f)^*\phi \big)\nonumber \\
&=\big(f^* \circ h^*(g),f^* \circ h^*(\phi) \big)\nonumber \\
&=\big(\mathcal{E}(f) \circ \mathcal{E}(h) (g),f^* \circ h^*(\phi) \big)\nonumber \\
&=\mathcal{E}(f)\circ \mathcal{E}(h) ((g,\phi)) \nonumber \\
&=\mathcal{G}((g,\phi)). \nonumber
\end{align*}

\newpage

The computation above implies the commutativity of the diagram 

\begin{equation*}
\begin{tikzpicture}
\matrix (m) [matrix of math nodes,row sep=3em,column sep=4em,minimum width=2em] {
	\mathcal{F}(g)  & \mathcal{F}(g') \\
	\mathcal{G}(g) & \mathcal{G}(g')\\};
\path[-stealth]
(m-1-1) edge  node [left] {{\small $ m_g $}} (m-2-1)
edge  node [above] {{\small $ \mathcal{F}((g,\phi)) $}} (m-1-2)
(m-2-1.east|-m-2-2) edge  node [below] {} node [below] {{\small $ \mathcal{G}((g,\phi)) $}} (m-2-2)
(m-1-2) edge  node [right] {{\small $m_{g'}$}} (m-2-2);
\end{tikzpicture}
\end{equation*}

 Furthermore, it is clear from Equation (\ref{strongcondition}) and from the construction that $\psi_{h,f}:\mathcal{E}(h \circ f)\Rightarrow \mathcal{E}(f) \circ \mathcal{E}(h)$ is in fact  invertible. In other words, we have $\mathcal{E}(h \circ f)\cong \mathcal{E}(f) \circ \mathcal{E}(h)$ up to  invertible natural transformation. 
 
 This completes the proof of $(i).$ 
\vspace{3pt}

\noindent \textit{Proof of \textit{(ii)}.} 
If $U\xrightarrow{f} V \xrightarrow{h} W$ in $\mathcal{C}$ is a composition, then we have
\begin{align}
& (1) \ \ \ \mathcal{F}(g)=\mathcal{G}(g) \ \ \text{for any} \ \ g \in Ob(\mathcal{E}(W)), \label{strong cond. objectwise}\\
&  (2) \ \ \ \mathcal{F}((g,\phi))=\mathcal{G}((g,\phi)) \ \text{for any} \ g \xrightarrow[(g,\phi)]{\sim}g' \ \ in \ \mathcal{E}(W), \label{strong cond. on morphisms}
\end{align} 
where $\mathcal{F}:=\mathcal{E}(h \circ f)$ and $ \mathcal{G}:=\mathcal{E}(f) \circ \mathcal{E}(h)$. 

Now, let $U\xrightarrow{f} V \xrightarrow{h} W \xrightarrow{p} Z$ be a composition of morphisms in $\mathcal{C}$, then it suffices to show that the associativity condition  holds both \textit{objectwise} and \textit{morphismwise}.
\begin{itemize}
	\item Let $g \in Ob(\mathcal{E}(Z))$, then we have 
	\begin{align*} 
	\mathcal{E}(p \circ (h \circ f))(g)&=\mathcal{E}(h \circ f)\circ\mathcal{E}(p)(g) \ \ \ from \ Eqn. (\ref{strong cond. objectwise}) \ with \ \psi_{p,h\circ f} \\
	&=\mathcal{E}(f) \circ \mathcal{E}(h) \circ \mathcal{E}(p) \ \ \ from \ Eqn. (\ref{strong cond. objectwise}) \ with \ \psi_{h, f} \star id_{\mathcal{E}(p)}\\
	&=\mathcal{E}(f) \circ \mathcal{E}(p \circ h) (g)\ \ \ from \ Eqn. (\ref{strong cond. objectwise}) \ with \ id_{\mathcal{E}(f)} \star \psi_{p, h}  \\
	&=\mathcal{E}((p \circ h) \circ f)(g) \ \ \ \ from \ Eqn. (\ref{strong cond. objectwise}) \ with \ \psi_{p\circ h, f}\nonumber
	\end{align*}This gives the commutativity of the diagram objectwise.
	
	\item Let $g \xrightarrow[(g,\phi)]{\sim}g'$ in $\mathcal{E}(Z)$, then we have 
	\begin{align*} 
	\mathcal{E}(p \circ (h \circ f))((g,\phi))&=\mathcal{E}(h \circ f)\circ\mathcal{E}(p)((g,\phi)) \ \ \ \ \ \ \ from \ Eqn. (\ref{strong cond. on morphisms}) \ with \ \psi_{p,h\circ f} \\
	&=\mathcal{E}(f) \circ \mathcal{E}(h) \circ \mathcal{E}(p)((g,\phi)) \ \ \ from \ Eqn. (\ref{strong cond. on morphisms}) \ with \ \psi_{h, f} \star id_{\mathcal{E}(p)}\\
	&=\mathcal{E}(f) \circ \mathcal{E}(p \circ h) ((g,\phi))\ \ \ \ \ \ \ from \ Eqn. (\ref{strong cond. on morphisms}) \ with \ id_{\mathcal{E}(f)} \star \psi_{p, h}  \\
	&=\mathcal{E}((p \circ h) \circ f)((g,\phi)) \ \ \ \ \ \ \ from \ Eqn. (\ref{strong cond. on morphisms}) \ with \ \psi_{p\circ h, f}. \nonumber
	\end{align*}	
\end{itemize} This completes the proof of $(ii)$, and hence that of Lemma \ref{main result on being pre-stack}.
\begin{flushright}
	$ \square $
\end{flushright}
\newpage

Let  
$ \mathcal{E}:\mathcal{C}^{op} \rightarrow  Grpds  $
  be the prestack defined in Lemma \ref{main result on being pre-stack}. 
  Now,   introducing a suitable site structure on $\mathcal{C}$, we  give 
the proof of Theorem \ref{THM 1}.


\vspace{3pt}

\paragraph{\bf Proof of Theorem \ref{THM 1}} \label{stacky} 
As in the case of \cite{Benini}, we first endow $\mathcal{C}$ with an appropriate Grothendieck topology $ \tau $ by defining the covering families $ \{U_i \rightarrow U \} $ of $U$ in $\mathcal{C}$ to be ``good" open covers $\{ U_i\subseteq U\}$ meaning that the fibered products 
$ U_{i_1 i_2 ... i_m}:=U_{i_1}\times_U U_{i_2} \times_U \cdot \cdot \cdot \times_U U_{i_m} $
 corresponding to the intersection of those open subsets $U_i$'s in $U$ are either empty or open subsets diffeomorphic to $\mathbb{R}^{n}$. Here each morphism  
$ U_i \hookrightarrow U $
is the canonical inclusion (and hence a morphism in $ \mathcal{C} $).  

Let $U$ be an object in $\mathcal{C}$. Given $\{U_i\subseteq U \} $ a covering family for $U$, one has the following cosimplicial diagram in $Grpds$
\begin{equation*}
\mathcal{E}(U_{\bullet}):= \bigg(\prod_i\mathcal{E}(U_i)\mathrel{\substack{\textstyle\rightarrow\\[-0.1ex]
		\textstyle\rightarrow \\[-0.1ex]}} \prod_{ij}\mathcal{E}(U_{ij})
\mathrel{\substack{\textstyle\rightarrow\\[-0.1ex]
		\textstyle\rightarrow \\[-0.1ex]
		\textstyle\rightarrow\\[-0.3ex]}} \prod_{ijk}\mathcal{E}(U_{ijk})
\mathrel{\substack{\textstyle\rightarrow\\[-0.1ex]
		\textstyle\rightarrow \\[-0.1ex]
		\textstyle\rightarrow \\[-0.1ex]
		\textstyle\rightarrow \\[-0.3ex]}} 
\cdot\cdot\cdot \bigg),
\end{equation*}  where  $U_{i_1 i_2 ...i_m}$ denotes the fibered product of $U_{i_n}$'s in $U$ as above. 
Note that for  a family \begin{equation*}
\{g_i  \} \ \ \text{in} \ \ \displaystyle \prod_i\mathcal{E}(U_i),
\end{equation*}  where $\mathcal{E}(U_i)=FMet(U_i)$, the coface maps $ d_0^1 $ and $d_1^1$ correspond to the suitable restrictions of each component
$ g_i |_{U_{ij}}  \text{ and }  g_j |_{U_{ij}}, $ respectively.

Now, it follows from the Lemma \ref{holim as groupoid} that $holim_{\small Grpds}(\mathcal{E}(U_{\bullet})) $ is indeed a particular groupoid and  can be defined as follows.
\begin{enumerate}
	\item \textit{Objects} are the pairs $  (x,h)  $, where 
$ 	x:= \{g_i \} \in \prod_i\mathcal{E}(U_i). $
That is, it is a family of Ricci-flat pseudo-Riemannian metrics on $ U_i $'s, along with the diagram

	\begin{equation*}
	\setlength{\unitlength}{0.9cm}
	\begin{picture}(12,4)
	\thicklines
	\put(9,3){\circle*{0.1}}
	\put(8.3,3.3){\small $\{g_k |_{U_{ijk}}\}$}
	\put(8,1){\circle*{0.1}}
	\put(6.7,0.4){\small $\{g_i |_{U_{ijk}}\}$}
	\put(10,1){\circle*{0.1}}
	\put(10.1,0.4){\small $\{g_j |_{U_{ijk}}\}$}
	\put(8,1){\line(1,0){2}}
	\put(8,1){\line(1,2){1}}
	\put(10,1){\line(-1,2){1}}
	\put(8.8,0.5){{\small $ \varphi_{ij} $}}
	\put(7.8,1.9){{\small $ \exists $}}
	\put(9.8,1.9){{\small $\varphi_{jk} $}}
	
	\put(1,2.2){$\bullet$}
	\put (0.4,1.6){$ \{g_i\} $}
	\qbezier(1.4,2.7)(2,3.5)(2.8,3.1)
	\put(2.8,3.1){\vector(1,0){0.1}}
	\qbezier(1.5,2)(2,1.1)(4.8,1.5)
	\put(4.8,1.5){\vector(1,1){0.1}}

	
	\put(3.1,3){\line(3,-2){2}}
	\put(4.1,2.5){\small$ \sim $}
	\put(3.1,3.2){\small$\{g_i |_{U_{ij}}\}$}
	\put(3.1,3){\circle*{0.1}}
	\put(5.2,1.90){\small $\{g_j |_{U_{ij}}\}$}
	\put(5.1,1.66){\circle*{0.1}}
	\qbezier(4.8,2.8)(6,4.5)(8,2.5)
	\put(8,2.5){\vector(1,-1){0.1}}
	\put(5.5,3.8){ \small $ d_1^2 $}
	
	\end{picture}
	\end{equation*}where $g_j |_{U_{ij}}= g_i |_{U_{ij}} \cdot \varphi_{ij}$ 
	for some $ \varphi_{ij}\in Aut(Sym^2(T^*U_{ij})). $ The ``triangle" on the RHS of  the diagram above implies that for all $i,j,k$, we have 
	\begin{align}
	g_k |_{U_{ijk}} & =  g_j |_{U_{ijk}} \cdot \varphi_{jk} \nonumber\\
	& = (g_i |_{U_{ijk}} \cdot \varphi_{ij}) \cdot \varphi_{jk} \nonumber\\
	& =  g_i |_{U_{ijk}} \cdot (\varphi_{jk} \circ \varphi_{ij}).
	\end{align}
	
It means that there exists a morphism $\varphi_{ik}: g_i |_{U_{ijk}} \xrightarrow{\sim}g_k |_{U_{ijk}}$. Therefore, we define the morphism $h$ in $ \prod\mathcal{E}(U_{ij}) $ as a family \begin{equation*}
\Big \{ g_i |_{U_{ij}} \xrightarrow[(g_i |_{U_{ij}},\varphi_{ij})]{\sim} g_j |_{U_{ij}} :  g_j |_{U_{ij}}=g_i |_{U_{ij}} \cdot \varphi_{ij} \ \& \ \varphi_{ij}\in Aut(Sym^2(T^*U_{ij}))  \Big \},
	\end{equation*}  where $  g_k |_{U_{ijk}} =  g_i |_{U_{(U_{ij}}} \cdot (\varphi_{jk} \circ \varphi_{ij}) $ and $s_0^0(h):\{g_i\}\rightarrow \{g_i\}$, which is just the identity morphism. 
	
	As a remark, the conditions in the definition of the family $ \{h\} $ correspond to  those in Lemma \ref{holim as groupoid} (Eqns. (\ref{key lemma objects property a}) and (\ref{key lemma objects property b})). Therefore, \textbf{an object of $holim_{\small Grpds}(\mathcal{E}(U_{\bullet})) $} is of the form \begin{equation*}
	(x,h)=\Big(  \{g_i \in FMet(U_i)\}, \{ \varphi_{ij} \in Aut(Sym^2(T^*U_{ij}))  \}       \Big),
	\end{equation*}
	where $ \{g_i\} $ is an object in $ \prod \mathcal{E}(U_i) $, and for each $i,j$, $ \varphi_{ij} := (g_i |_{U_{ij}},\varphi_{ij}) $  is a morphism in $ \prod\mathcal{E}(U_{ij}) $ satisfying 
	\begin{align*}
	& (i) \ \  g_j |_{U_{ij}}=g_i |_{U_{ij}} \cdot \varphi_{ij},  \ \text{with} \ \varphi_{ij}\in Aut(Sym^2(T^*U_{ij})) ,\\
	& (ii) \ \  On \ U_{ijk}, \ \ \varphi_{ij}\circ\varphi_{jk} = \varphi_{ik} \ \ (the \ cocycle \ condition),\\
	& (iii) \ \ s_0^0(h):\{g_i\}\rightarrow \{g_i\}, \ the \ identity \ morphism. 
	\end{align*}
	
	In short, an object $\mathrm{\textbf{g}}:= \big(\{g_i\},\{\varphi_{ij}\}\big)$ in $holim_{\small Grpds}(\mathcal{E}(U_{\bullet}))$ is a collection $ \{g_i\} $ of Ricci-flat metrics over covering open subset $U_i$ of $U$, together with the transition maps  $ \{ \varphi_{ij}\}$ on the overlaps that satisfy the cocycle condition above.
	\vspace{3pt}
	
	\item \textit{A morphism} $ (x,h)\rightarrow(x',h') $ in $holim_{\small Grpds}(\mathcal{E}(U_{\bullet})) $
	 consists of the following data:
	
	\begin{enumerate}
		\item A morphism $x\xrightarrow{f}x'$ in $ \prod \mathcal{E}(U_i) $, such that 
	$ 	\{g_i\} \xrightarrow{\sim} \{g_i'\}, $
		where $ g_i,g_i' \in FMet(U_i)$ with $g_i'= g_i \cdot \varphi_i $ for some $\varphi_i \in Aut(Sym^2(T^*U_{i})).$
		\item For each $i,j$, a commutative diagram 
		\begin{equation} \label{morphisims are compatible with the trans maps}
		\begin{tikzpicture}
		\matrix (m) [matrix of math nodes,row sep=3em,column sep=4em,minimum width=2em] {
			g_i |_{U_{ij}} & g_i' |_{U_{ij}}\\
			g_j |_{U_{ij}} & g_j' |_{U_{ij}}\\};
		\path[-stealth]
		(m-1-1) edge  node [left] {{\small $ h=\varphi_{ij}$}} (m-2-1)
		edge  node [above] {{\small $\varphi_i|_{U_{ij}} $}} (m-1-2)
		(m-2-1.east|-m-2-2) edge  node [below] {} node [below] {{\small $ \varphi_j|_{U_{ij}} $}} (m-2-2)
		(m-1-2) edge  node [right] {{\small $h'=\varphi_{ij}'$}} (m-2-2);
		\end{tikzpicture}
		\end{equation}
		In fact, it follows from the fact that $ g_j |_{U_{ij}}= g_i |_{U_{ij}} \cdot \varphi_{ij}$ and $ g_j' |_{U_{ij}}= g_i' |_{U_{ij}} \cdot \varphi_{ij}'$, we have 
		\begin{align*}
		\big(g_i |_{U_{ij}} \cdot \varphi_i|_{U_{ij}}\big) \cdot \varphi_{ij}' =  g_j' |_{U_{ij}}. 
		\end{align*}
	On the other hand, one also has 
		\begin{align*}
		\big(g_i |_{U_{ij}} \cdot \varphi_{ij}\big) \cdot \varphi_j|_{U_{ij}}  =  g_j' |_{U_{ij}},
		\end{align*}
		which imply the commutativity of the diagram, and hence one can also deduce the following relation:
		\begin{align*}
		&\big(g_i |_{U_{ij}} \cdot \varphi_{ij}\big) \cdot \varphi_j|_{U_{ij}}    =  \big(g_i |_{U_{ij}} \cdot \varphi_i|_{U_{ij}}\big) \cdot \varphi_{ij}' \ \ \forall i,j \\
		& \Longleftrightarrow \nonumber\\
		& g_i |_{U_{ij}} \cdot\big( \varphi_j|_{U_{ij}} \circ \varphi_{ij} \big)   =  g_i |_{U_{ij}} \cdot \big(\varphi_{ij}'\circ \varphi_i|_{U_{ij}}\big) \ \ \forall i,j \\
		& \Longleftrightarrow \nonumber\\
		& \varphi_{ij}'= \varphi_j|_{U_{ij}} \circ \varphi_{ij} \circ  \varphi_i^{-1}|_{U_{ij}} \ \ \forall i,j
		\end{align*}

	\end{enumerate}
	
\end{enumerate}

\noindent Thus,  \textbf{a morphism} in $holim_{\small Grpds}(\mathcal{E}(U_{\bullet})) $ from $ \textbf{g}=\big(\{g_i\},\{\varphi_{ij}\}\big) $ to $\textbf{g}'= \big(\{g_i'\},\{\varphi_{ij}'\}\big) $ is a family 
\begin{equation}\label{relation for being a morphism}
\Big\{\varphi_i \in Aut(Sym^2(T^*U_{i})) : g_i'=g_i\cdot\varphi_i \ \& \ \varphi_{ij}'= \varphi_j|_{U_{ij}} \circ \varphi_{ij} \circ  \varphi_i^{-1}|_{U_{ij}} \Big\}
\end{equation} 

In short, a morphism $ \boldsymbol{\varphi}: \textbf{g}\rightarrow\mathrm{\textbf{g}}'$ in $holim_{\small Grpds}(\mathcal{E}(U_{\bullet})) $ is a collection $ \{\varphi_{i} \} $ of morphisms, with $\varphi_{i}\in Mor_{\mathcal{E}(U_i)}(g_i, g'_i)$,  such that the action is compatible with the corresponding transition maps in the sense of Diagram \ref{morphisims are compatible with the trans maps}.  

Now, for a covering family $\{U_i\subseteq U \}$  of $U$,  \textit{the canonical morphism}  
\begin{equation}\label{canonical morphism}
\Psi:\mathcal{E}(U)\longrightarrow holim_{\small Grpds}(\mathcal{E}(U_{\bullet}))
\end{equation} is defined as a functor of groupoids, where 
\begin{itemize}
	\item for each object $g$ in $FMet(U)$, \begin{equation*}
	g\xrightarrow{\Psi} \Big( \{g|_{U_i}\}, \{\varphi_{ij}=id \}\Big),
	\end{equation*} together with the trivial cocyle condition.
	\item for each morphism $g\xrightarrow[(g,\varphi)]{\sim}g\cdot\varphi$, with $\varphi\in Aut(Sym^2(T^*U_{ }))$, \begin{equation*}
	\big(g\xrightarrow[(g,\varphi)]{\sim}g\cdot\varphi\big) \xrightarrow{\Psi} \Big( \{\varphi_i:=\varphi|_{U_i} \}\Big),
	\end{equation*} where $\varphi|_{U_i}$ trivially satisfies the desired relation in Equation (\ref{relation for being a morphism}) for being a morphism in $holim_{\small Grpds}(\mathcal{E}(U_{\bullet})) $.
\end{itemize}
\begin{lemma} \label{lemma_fullyfaithfullandesstiallysurject}
$\Psi$ is a fully faithful and essentially surjective functor.
\end{lemma} 
 \pf

 $\Psi$ is \textit{essentially surjective}: Let $\mathrm{\textbf{g}}:= \big(\{g_i\},\{\varphi_{ij}\}\big)$ be an object in $holim_{\small Grpds}(\mathcal{E}(U_{\bullet}))$. Then we have a family of objects $ \{g_i \} $,  with the family of transition functions $\{\varphi_{ij}\}$ 
  satisfying the cocycle condition $ \varphi_{ij}\circ\varphi_{jk} = \varphi_{ik}   \text{ on }  U_{ijk}$, such that $$ g_j |_{U_{ij}}=g_i |_{U_{ij}} \cdot \varphi_{ij}.$$ 

We need to show that these are patched together to form a metric $g\in FMet(U)$. In fact,  our site structure on $\mathcal{C}$ consists of good covers for which the intersection of open subsets $U_i$'s in $U$ are either empty or open subsets diffeomorphic to $\mathbb{R}^{n}$. Also, $Sym^2(T^*U_{ })$ is a locally free sheaf over $U$. In this regard, the following fact is useful: \textit{All cocycles are trivializable on manifolds diffeomorphic to $\mathbb{R}^n$.}
	Therefore, we conclude that $ \{ \varphi_{ij} = id\} $ for all $i,j$. 
	
	Now, we have a trivial cocycle condition with $\varphi_{ij}=id$. It follows that $g_i$ is a section of the sheaf $ Sym^2(T^*U_{i}) $ over $U_i$ satisfying $g_j |_{U_{ij}}=g_i |_{U_{ij}}$ for all $i,j$. So, $g_i$'s are glued together by transition functions $\varphi_{ij}$, along with the trivial cocycle condition, to form $g\in FMet(U)$ so that $g|_{U_i}=g_i$ and $ \varphi|_{U_i} = \varphi_i $ for all $i.$ Therefore, $\Psi$ is \textit{essentially surjective.}
	\vspace{3pt}
	
$\Psi$ is \textit{fully faithful:} We need to show that the induced map \begin{equation*}
	\underline{\Psi}: Hom_{\mathcal{E}(U)}(g,g')\longrightarrow Hom_{holim_{\small Grpds}(\mathcal{E}(U_{\bullet}))}(\Psi(g),\Psi(g'))
	\end{equation*} is a bijection of sets. To this end, we consider the corresponding \emph{sheaf-Hom} $\mathcal{H}om(\mathcal{S}, \mathcal{S'})$, with $\mathcal{S}= \mathcal{S'}=Sym^2(T^*U_{ })$, where $\mathcal{H}om(\mathcal{S}, \mathcal{S'})$ is the collection of the data
	\begin{equation*}
	\mathcal{H}om(\mathcal{S}, \mathcal{S'})(V):=Mor(\mathcal{S} |_V, \mathcal{S'}|_V).
	\end{equation*} Here $ \mathcal{S} |_V $ denotes the restriction of the sheaf to the open subset $V\subset U$. Then, both injectivity and surjectivity of $ \underline{\Psi}$ follow from the fact that the \emph{sheaf-Hom} $\mathcal{H}om(\mathcal{S}, \mathcal{S'})$ is a sheaf over $U$.  Let us explain the details below.
	
	If we assume $\underline{\Psi}(\varphi)=id$, then  it means, by definition, $\varphi_i:=\varphi|_{U_i} = id$ for all $i$. By construction, it implies that each $\varphi|_{U_i} \in \mathcal{H}om(Sym^2(T^*U_{ }), Sym^2(T^*U_{ }))(U_i)$. Because \emph{sheaf-Hom} is a sheaf over $U$, we obtain $\varphi=id$, and hence injectivity of $ \underline{\Psi}$. 
	
	Now, let  $ \boldsymbol{\varphi}: \Psi(g)\rightarrow\Psi(g')$ be a a morphism in $holim_{\small Grpds}(\mathcal{E}(U_{\bullet}))$. Then it can be viewed as  a collection $ \{\varphi_{i} \} $ of morphisms  such that the action is compatible with the corresponding transition maps in the sense of Diagram \ref{morphisims are compatible with the trans maps}. Here both $ \Psi(g) $ and $\Psi(g')$ are collections of Ricci-flat metrics $ \{g|_{U_i}\} $ and $ \{g'|_{U_i}\}$, respectively, along with the trivial transition maps. Therefore, Diagram \ref{morphisims are compatible with the trans maps} with $\varphi_{ij}=\varphi'_{ij}=id$ implies that $\varphi_{j}|_{U_{ij}}=\varphi_{i}|_{U_{ij}}$, where each $ \varphi_{i} \in \mathcal{H}om(Sym^2(T^*U_{ }), Sym^2(T^*U_{ }))(U_i)$. Because \emph{sheaf-Hom} is a sheaf over $U$, we conclude that there exists $\varphi \in \mathcal{H}om(Sym^2(T^*U_{ }), Sym^2(T^*U_{ }))(U)$ such that $\varphi|_{U_i}=\varphi_{i}$. Equivalently, it means $\varphi \in Hom_{\mathcal{E}(U)}(g,g')$, with $ \underline{\Psi}(\varphi)= \{\varphi_{i} \}$. This proves  the desired surjectivity and  
	completes the proof. 
	\epf

From Lemma \ref{lemma_fullyfaithfullandesstiallysurject}, we conclude that the canonical morphism $\Psi$  in Equation (\ref{canonical morphism}) is a weak equivalence in $Grpds$, and this completes the proof of Theorem \ref{THM 1}. 
\begin{flushright}
	$ \square $
\end{flushright}

\begin{definition} \label{defn_moduli stack of GR}
	The stack $ \mathcal{E}:\mathcal{C}^{op} \rightarrow  Grpds $ constructed above
is called the \textit{moduli stack of solutions to the vacuum Einstein field equations on $M$},  with $\Lambda=0$. We sometimes call it directly the \emph{stack of Einstein gravity}.
\end{definition} 

\subsection{Proof of Theorem \ref{THM 1.1}} \label{section_proof on families}
In this section, we provide a sketch of the proof of Theorem \ref{THM 1.1}. In fact, after fixing our  notation and giving the explicit definitions, the result follows from Theorem \ref{THM 1} with some natural modifications.

\begin{definition}
	Denote by $Cart$ the category of cartesian spaces, where an object is an open subset of $\mathbb{R}^n$ that is diffeomorphic to $\mathbb{R}^n$, and morphisms are smooth maps. To turn $Cart$ into a site, we declare a cover of an object $U$ to be ``good covers" $\{U_i \rightarrow U \}$, i.e., open covers for which every intersection of those open subsets $U_i$'s in $U$ is either empty or diffeomorphic to $\mathbb{R}^{n}$. 
\end{definition} We are particularly interested in the site $Cart$ because objects in $Fun(Cart^{op}, \mathcal{C})$ can be viewed as the category of smoothly parametrized objects of $\mathcal{C}$ over cartesian spaces. As an example, any manifold $M$ can be considered as a functor $$Cart^{op} \rightarrow Sets, \ U \mapsto C^{\infty}(U,M).$$ 

Note also that in the proof of Theorem \ref{THM 1}, morphisms in the source category are all canonical inclusions, and hence pullbacks of  (Ricci-flat) metrics by these morphisms are just restrictions to some smaller open subsets, and hence still Ricci-flat. Therefore, for a ``family version" of this category, (fiberwise) open embeddings can be viewed as suitable substitutes. Moreover, we require our geometric structure (Lorentzian with Ricci-flatness) to vary in families parametrized over cartesian spaces \footnote{More details on geometric structures via stacks and on geometries in families can be found in \cite{GradyPavlov,LudewigStoffel}}.  

Therefore, throughout this subsection, we  work with sheaves on the site $ Fam_n $ of families of manifolds with $n$-dimensional fibers, together with fiberwise open embeddings. More precisely, we have

\begin{definition}
	Let $ Fam_n $ be the site, where an object, denoted by $M/U$, is a submersion $\pi:M\rightarrow U$ with $n$-dimensional fibers and $U$ an object in $Cart$, and a morphism $M/U \rightarrow M'/U'$ is a smooth bundle map that is a fiberwise open embedding.
	 
	Moreover, the site structure is determine by the covering families that are a collection of morphisms $\{ M_i/U_i\rightarrow M/U\}$ such that $\{M_i\}$ is an  open cover of $M.$ 
\end{definition}

\paragraph{\textbf{A sketch of the proof of Theorem \ref{THM 1.1}}} Denote by $\mathcal{E}^{fam}$ the presheaf on $Fam_n$	\begin{equation*} 
Fam_n^{op} \longrightarrow Grpds, \ \ M/U \mapsto  \mathcal{E}^{fam} (M/U),
\end{equation*} where $ Ob(\mathcal{E}^{fam} (M/U)):=\{g\in \Gamma(Sym^2(T^*(M/U))) :  Ric(g)=0 \} $.

Here $ T^*(M/S)$ is the relative cotangent bundle $ T^*(M/U)=Coker(T^*U\rightarrow T^*M)$, which allows us to define fiberwise versions (or ``families") of many familiar structures. Indeed, we are currently interested in (pseudo) Riemannian structures. In this regard, a pseudo-Riemannian metric $g$ on $M/U$ is a section of the relative bundle $ Sym^2(T^*(M/U)).$ In other words, for an object $\pi: M \rightarrow U$ in $Fam_n$, $g$ is a (Ricci-flat) pseudo-Riemannian metric on the vertical tangent bundle $ker(\pi_*) \subset TM$. Thus, for any parameter $u\in U$ and $p\in M_u:= \pi^{-1}(u)$, $g|_p$ is a  metric on $ker(\pi_{*,p}) \subset T_pM$.

Using the fact that an object of $ \mathcal{E}^{fam} (M/U) $ is a (Ricci-flat) pseudo-Riemannian metric on the vertical tangent bundle $ker(\pi_*) \subset TM$, \emph{morphisms} in the groupoid $ \mathcal{E}^{fam} (M/U) $ can be defined via particular automorphisms of $ Sym^2(T^*(M/U))$ induced by infinitesimal transformations as in Lemma \ref{main result on being pre-stack}. Likewise, composition can be defined by using similar arguments in Lemma \ref{main result on being pre-stack}.

Functoriality follows from the fiberwise nature of the current construction. Given a morphism $F: N/V \rightarrow M/U$ in $Fam_n$, we have a commutative diagram
\begin{equation*}
\begin{tikzpicture}
\matrix (m) [matrix of math nodes,row sep=3em,column sep=4em,minimum width=2em] {
	N  & M \\
	V & U\\};
\path[-stealth]
(m-1-1) edge  node [left] {{\small $\pi_V $}} (m-2-1)
edge  node [above] {{\small $ F $}} (m-1-2)
(m-2-1.east|-m-2-2) edge  node [below] {} node [below] {{\small $ f $}} (m-2-2)
(m-1-2) edge  node [right] {{\small $\pi_U$}} (m-2-2);
\end{tikzpicture}
\end{equation*} such that for each $v\in V $, $F_v: N_v\rightarrow M_{f(v)} $ is an open embedding. If $g$ is a Ricci-flat metric on $T(M/U)$, so is its pullback under fiberwise open embeddings. Therefore, using the diagram above, $F^*g$ gives a Ricci-flat metric on $T(N/V)$, and hence an object in   $ \mathcal{E}^{fam} (N/V). $ Likewise, a morphism $ \phi $ in $ \mathcal{E}^{fam} (M/U) $ can be pulled-back via $F$, and due to the fiberwise action of the morphisms, $F^*\phi$ gives a morphism in $ \mathcal{E}^{fam} (N/V).$ The other compatibility conditions are straightforward to check by following similar arguments in Lemma \ref{main result on being pre-stack}.

Finally, one can achieve the stackification of the prestack $ \mathcal{E}^{fam} $ by following more or less the same arguments in the proof of Theorem \ref{THM 1} ``fiberwisely", with some modifications (using families, fiberwise open embeddings, and the site structure above, etc.).

\subsection{Proof of Theorem \ref{THM-2}} \label{stacky equivalence}
As discussed before, the \textit{equivalence} between quantum gravity and  gauge theory holds if the phase spaces of GR and the associated gauge theory can be identified (cf. Definition \ref{equvalence}). In fact, Mess proved \cite{Mess} that this is possible for a particular setup  (cf. Theorem \ref{thm_Mess reult summary}).

Now, we would like  to show that once it exists, the equivalence induces an isomorphism between the corresponding moduli stacks. To this end, we shall first revisit \cite{Benini} and introduce a particular stack similar to $BG_{con}$ given in \cite[Example 2.11]{Benini}.  This helps us to view the space $\mathcal{M}_{flat}$ as a certain stack.

 Of course, we first need to introduce the ``flat" counterpart of this classifying stack $BG_{con}$ in a na\"{\i}ve way. Just for  simplicity, we use $\mathcal{M}$ for the flat case whose construction is the same as that of  $BG_{con}$. 
Keep also in mind that for the gravitational interpretation (with $\Lambda=0$), one requires to consider the case of $G=ISO(2,1)$. In this regard, we have the following lemma.

\begin{lemma} \label{flatversion of BG}
	Let $ \mathcal{C} $ be the category in Lemma \ref{main result on being pre-stack} such that $M$ is a Lorentzian 3-manifold topologically of the form $\Sigma\times \mathbb{R}$ with $\Sigma$ a closed Riemann surface of genus $g>1$. The  functor $ \mathcal{M}:\mathcal{C}^{op} \rightarrow \ Grpds $ described below     is a stack.
	
	\begin{enumerate}
		\item For each object $U$ of $\mathcal{C}$, $ \mathcal{M}(U) $ is a groupoid of flat $G$-connections on $U$, where objects are the elements of the \textit{set} 
		$ \Omega^1(U, \mathfrak{g})_{flat} $ of Lie algebra-valued 1-forms on $U$, with $ F_{A}=0 $, and morphisms form the set 
		\begin{equation*}
		Hom_{\mathcal{M}(U)}(A,A')=\{ \rho \in \mathcal{G}:A'=A\bullet \rho \},
		\end{equation*} 
		where the action of the gauge group $ \mathcal{G} $, which is locally of the form $C^{\infty}(U,G)$, on $\Omega^1(U, \mathfrak{g})_{flat}$  is defined as follows: 
		For all $\rho\in \mathcal{G}$ and $A \in \mathcal{A}$, we set  \begin{equation*}
		A \bullet \rho := \rho^{-1}\cdot A \cdot \rho + \rho^{-1}\cdot \mathrm{d} \rho.
		\end{equation*}
		We denote a morphism  $
		A\xrightarrow{\sim} A'=A\bullet \rho
	$  in $ Hom_{\mathcal{M}(U)}(A,A')$ by $ (A,\rho). $
		\vspace{3pt}
		
		\item To each morphism $ U \xrightarrow{f} V $ in $\mathcal{C}$, i.e. $f:U \hookrightarrow V$ with $U\subset V$, one assigns
	$ 	\mathcal{M}(V)\xrightarrow{\mathcal{M}(f)}\mathcal{M}(U). $		
		Here $ \mathcal{M}(f) $ is a functor of categories whose action on objects and on morphisms of $\mathcal{M}(V)$ is given as follows.
		\begin{enumerate}
			\item For any object $A\in \mathcal{M}(V)=\Omega^1(V, \mathfrak{g})_{flat}$,  we have
		$ 	A \xrightarrow{\mathcal{M}(f)} f^*A \ (=A|_U) $,
			where $ \mathcal{M}(f)(A):= f^*A \in \Omega^1(U, \mathfrak{g})_{flat}. $ Here we use the fact that the pullback (indeed the restriction to an open subset $U$ in our case) of a flat connection in the sense that $F_A=0$ is also flat. 
			\item For any morphism $ (A,\rho) \in Hom_{\mathcal{M}(V)}(A,A') $ with $\rho \in \mathcal{G}$ such that $ A'=A \bullet \rho $,  it follows from the fact that 
			\begin{equation} \label{pullback of gauged connection}
			f^*(A \bullet \rho)=f^*A \bullet f^*\rho,
			\end{equation} where $f^*\rho=\rho \circ f \in C^{\infty}(U,G)$, we conclude that $ f^*(A \bullet \rho) $ lies in the orbit space of $f^*A$. Hence we get   \begin{equation*}
			\bigg(A\xrightarrow[ (A,\rho) ]{\sim}  A'=A \bullet \rho\bigg) \xrightarrow{\mathcal{M}(f)} \bigg( f^*A \xrightarrow [(f^*A,\rho \circ f)]{\sim} f^*(A\bullet \rho)=f^*A \bullet f^*\rho  \bigg),
			\end{equation*} where $ \mathcal{M}(f)(A,\rho):=(f^*A,f^*\rho) $ is a morphism in $\mathcal{M}(U). $ Note that Equation (\ref{pullback of gauged connection})  can indeed be proven by just \textit{local computations} of the pullback of a connection $ A $ together with the action $A \bullet \rho.$
		\end{enumerate}
		
	\end{enumerate}
	
\end{lemma}
\newpage
\begin{proof}
	This is similar to the proofs of Lemma \ref{main result on being pre-stack} and Theorem \ref{THM 1}, with the special setup, where $n=3$ and $M$ as above. For a complete treatment to the generic case (i.e. without flatness requirement), see \cite[Examples 2.10 and 2.11]{Benini}.  For the flat case, on the other hand, one has exactly the same proof with $ \Omega^1(U, \mathfrak{g})_{flat} $ instead of $ \Omega^1(U, \mathfrak{g}) $ thanks to the fact that the pullback  of a flat connection by a canonical inclusion  $U\hookrightarrow V$ between open subsets is also flat. 
\end{proof}

Let us summarize our progress so far. 

\begin{enumerate}
	\item Before  stacky constructions, we already have an isomorphism of phase spaces 
	$ \phi:\mathcal{E}(M)\xrightarrow{\sim} \mathcal{M}_{flat} $
	 in the case of vacuum Einstein gravity, with the cosmological constant $\Lambda=0$,  on a Lorentzian 3-manifold $M=\Sigma\times (0,\infty)$, where $\Sigma$ is a closed Riemann surface of genus $g>1$.
	\item We define the \textit{moduli stack $\mathcal{E}$ of Einstein gravity}  (cf. Definition \ref{defn_moduli stack of GR}). 

	\item From Lemma \ref{flatversion of BG}, we introduce the \textit{classifying stack $ 	\mathcal{M} $ of principal G-bundles with flat connections} on $\Sigma$,
	where $\mathcal{C}$, in that case, involves  particular choices of dimension ($n=3$) and the form  $M$ $ :=\Sigma\times (0,\infty) $.	 
\end{enumerate}

Given  a closed Riemann surface $\Sigma$ of genus $g>1$, we now intend to show that if $\mathcal{C}$ is the category in Lemma \ref{main result on being pre-stack}, with  $M$ a Lorentzian 3-manifold of the form $\Sigma\times (0,\infty)$, then there exists \textit{an invertible natural transformation}
\begin{equation*}
{\small \begin{tikzpicture}
	\matrix[matrix of nodes,column sep=1.5cm] (cd)
	{
		$ \mathcal{C}^{op} $ & $ Grpds, $ \\
	};
	\draw[->] (cd-1-1) to[bend left=50] node[label=above:$ {\small \mathcal{E}} $] (U) {} (cd-1-2);
	\draw[->] (cd-1-1) to[bend right=50,name=D] node[label=below:${\small \mathcal{M}}$] (V) {} (cd-1-2);
	\draw[double,double equal sign distance,-implies,shorten >=10pt,shorten <=10pt] 
	(U) -- node[label=left:$\small \Phi$] {} (V);
	\end{tikzpicture}}
\end{equation*}
between these two stacks $\mathcal{E}$ and $\mathcal{M}$. This eventually provides a stacky extension of the isomorphism between the corresponding classical phase spaces.

\vspace{3pt}

\paragraph{\bf Proof of Theorem \ref{THM-2}}  
From the gauge theoretic realization of 3D gravity (with $\Lambda=0$) in Cartan's formalism, any solution to the vacumm Einstein field equations, with $\Lambda=0$, on any open subset of $M$ defines  a flat $ISO(2,1)$-connection. Thus, for any object $U$ in $\mathcal{C}$, we have a natural map
\begin{equation}\label{definintion of objectwise map}
\Phi_U: \mathcal{E}(U)\longrightarrow \mathcal{M}(U),
\end{equation}
which is indeed \textit{a functor of groupoids} defined as follows: 
\begin{enumerate}
	\item To each $g\in FMet(U)$, 
	 one assigns the corresponding flat $ISO(2,1)$-connection $A^g$ in $ \Omega^1(U, \mathfrak{iso(2,1)})_{flat} $ described by  Cartan's formalism. That is,
	\begin{equation} \label{a connection associated to the metric}
	g \xrightarrow{\Phi_U} A^g.
	\end{equation}
	\item As mentioned before, Cartan's formalism 
	encodes the symmetries of each theory in the sense that \textit{the diffeomorphism invariance of 3D gravity theory does correspond to the gauge invariance behaviour of the associated Chern-Simons theory (and vice versa)} \cite{Wit1}. It means that  equivalence classes $[g]$ of  flat pseudo-Riemannian metrics correspond to the gauge equivalence classes of the associated connections $[A^g]$. From Remark \ref{remark_the equivalence between diffeomorphisms and  gauge transformations}, we only consider diffeomorphisms in the connected component of the identity to ensure the desired equivalence. 
	
	In brief, for any $g'\in [g]$ over an open subset $U$, i.e. $g'=g\cdot \varphi$ for some automorphism $\varphi$ of $Sym^2(T^*U)$, the corresponding connections 
	$A^g  \text{ and }  A^{g\cdot \varphi}
	$ are also gauge equivalent, and hence lie in the same equivalence class (and vice versa). That is, there exists $\rho_{\varphi} \in \mathcal{G}$, an infinitesimal gauge transformation associated to $\varphi$, such that 
	$ A^{g\cdot \varphi} = A^g \bullet \rho_{\varphi}.$
	 In other words, such a correspondence can also be expressed as the  commutative diagram
	\begin{equation} \label{a gauge trans associated the diffeo}
	\begin{tikzpicture}
	\matrix (m) [matrix of math nodes,row sep=3em,column sep=3em,minimum width=2em] {
		g  & g\cdot \varphi \\
		A^g & A^{g\cdot \varphi} \\};
	\path[-stealth]
	(m-1-1) edge  node [left] {\small$ \Phi_U  $} (m-2-1)
	(m-1-1.east|-m-1-2)edge  node [above] {{\small $ \varphi $}} (m-1-2)
	(m-2-1.east|-m-2-2) edge  node [below] {} node [below] {{\small $ \rho_{\varphi} $}} (m-2-2)
	(m-1-2) edge  node [right] {\small $ \Phi_U  $} (m-2-2);
	\end{tikzpicture}
	\end{equation}
	together with two morphisms (relating infinitesimal diffeomorphisms and infinitesimal gauge transformations)
	\begin{align} \label{diffeo and gauge trans correspondance}
	& Aut(Sym^2(T^*U)) \longrightarrow C^{\infty}(U,G), \ \ \ \varphi \mapsto \rho_{\varphi}, \\
	& C^{\infty}(U,G) \longrightarrow Aut(Sym^2(T^*U)), \ \ \ \rho \mapsto \varphi_{\varphi}. 
	\end{align}Note that $ Aut(Sym^2(T^*U)) $ is endowed with the usual composition, and the group operation on $ C^{\infty}(U,G) $ is given by the  pointwise multiplication. 
	\item To each morphism $(g,\varphi):g\longrightarrow g'$ in $Hom_{\mathcal{E}(U)}(g,g')$, $ \Phi_U  $ associates a morphism 
	\begin{equation} \label{morphismswise correspondence}
	A^g \xrightarrow[(A^g,\rho_{\varphi})]{\sim}A^g \bullet \rho_{\varphi} \ (=A^{g'}),
	\end{equation} 
	where $\rho_{\varphi} \in C^{\infty}(U,ISO(2,1))$ is a gauge transformation corresponding to $\varphi$ in accordance with Diagram \ref{a gauge trans associated the diffeo}. Therefore, for any morphism $f:U\hookrightarrow V$ in $\mathcal{C}$, using the map in (\ref{diffeo and gauge trans correspondance}), one also has the following commutative diagram.
	\begin{equation} \label{pullback and correspondence}
	\begin{tikzpicture}
	\matrix (m) [matrix of math nodes,row sep=3em,column sep=4em,minimum width=2em] {
		Aut(Sym^2(T^*V))  & Aut(Sym^2(T^*U)) \\
		C^{\infty}(V,ISO(2,1)) & C^{\infty}(U,ISO(2,1))\\};
	\path[-stealth]
	(m-1-1) edge  node [left] {{\small $   $}} (m-2-1)
	edge  node [above] {{\small $ f^* $}} (m-1-2)
	(m-2-1.east|-m-2-2) edge  node [below] {} node [below] {{\small $ f^* $}} (m-2-2)
	(m-1-2) edge  node [right] {{\small $   $}} (m-2-2);
	\end{tikzpicture}
	\end{equation}

	\item \textit{Functoriality}. Given a composition of morphisms 
	\begin{equation*}
	{\small
		\begin{tikzpicture}[scale=4]
		\node  (a)  {$g$};
		\node  (b) [right=of a] {$g\cdot\varphi$};
		\node  (c) [right=of b] {$(g\cdot\varphi)\cdot \psi$,};
		\node (midpoint) at ($(a)!.4!(c)$) {};
		\node (psi) [above=.4cm of midpoint] {$(g \cdot \varphi)\cdot \psi = g \cdot (\psi \circ \varphi)$};
		\draw[->]  (a) -- node [below] {$ \varphi $} (b);
		\draw[->] (b) -- node [below] {$ \psi $} (c);
		\draw[->] (a.north east) parabola bend (psi.south) (c.north west);
		\end{tikzpicture}
	}
	\end{equation*}
	we have the following commutative diagram 
	\begin{equation*}
	\begin{tikzpicture}
	\matrix (m) [matrix of math nodes,row sep=3em,column sep=4em,minimum width=2em] {
		g & g \cdot \varphi & (g \cdot \varphi)\cdot \psi = g \cdot (\psi \circ \varphi)\\
		A^g & A^{g \cdot \varphi} & A^{g \cdot (\psi \circ \varphi)},\\};
	\path[-stealth]
	(m-1-1) edge  node [left] {{\small $ $}} (m-2-1)
	(m-1-1.east|-m-1-2) edge  node [above] {{\small $ \varphi $}} (m-1-2)
	(m-1-2) edge  node [above] {{\small $\psi$}} (m-1-3)
	(m-2-1.east|-m-2-2) edge  node [above] {{\small $ \rho_{\varphi} $}} (m-2-2)
	(m-2-2) edge  node [above] {{\small $\rho_{\psi}$}} (m-2-3)
	(m-1-3) edge  node [right] {{\small $ $}} (m-2-3)
	(m-1-2) edge  node [right] {{\small $ $}} (m-2-2);
	\end{tikzpicture}
	\end{equation*}
	where the vertical maps are $\Phi_U$, and using the commutativity, \begin{equation*}
	A^g \bullet \rho_{\psi \circ \varphi} = A^{g \cdot (\psi \circ \varphi)}= A^{(g \cdot \varphi)\cdot \psi}=A^{g\cdot\varphi} \bullet \rho_{\psi}=(A^g \bullet \rho_{\varphi})\bullet \rho_{\psi}.
	\end{equation*}Then 
	 we obtain $ A^g \bullet \rho_{\psi \circ \varphi}=A^g \bullet (\rho_{\varphi}\rho_{\psi}) $, and hence $\rho_{\varphi} \rho_{\psi} =\rho_{\varphi\circ \psi},
	$. This gives the desired functoriality: \begin{equation*}
	\Phi_U(g,\varphi\circ \psi)=\rho_{\varphi\circ \psi}=\rho_{\varphi}\rho_{\psi}=\Phi_U(g,\varphi)  \Phi_U(g,\psi).
	\end{equation*} 
	
\end{enumerate}

Now, we need to show that for each morphism $f:U\rightarrow V$ in $\mathcal{C}$, i.e. $f:U \hookrightarrow V$ with $U\subset V$, we have the following commutative diagram.
\begin{equation} \label{desired commutative diagram}
\begin{tikzpicture}
\matrix (m) [matrix of math nodes,row sep=3em,column sep=4em,minimum width=2em] {
	\mathcal{E}(V)  & \mathcal{M}(V) \\
	\mathcal{E}(U) & \mathcal{M}(V)\\};
\path[-stealth]
(m-1-1) edge  node [left] {{\small $ \mathcal{E}(f) $}} (m-2-1)
edge  node [above] {{\small $ \Phi_V $}} (m-1-2)
(m-2-1.east|-m-2-2) edge  node [below] {} node [below] {{\small $ \Phi_U $}} (m-2-2)
(m-1-2) edge  node [right] {{\small $\mathcal{M}(f)$}} (m-2-2);
\end{tikzpicture}
\end{equation}

In fact, the commutativity follows from the definition of $\Phi_U$: Let $g \in FMet(V)$, then we get, from the construction and from the restriction functor $\cdot |_U$, 
the natural diagram 
\begin{equation} \label{restriction diagram}
\begin{tikzpicture}
\matrix (m) [matrix of math nodes,row sep=2em,column sep=3em,minimum width=2em] {
	g  & A^g \\
	g|_U & A^{g|_U} = A^g|_U.\\};
\path[-stealth]
(m-1-1) edge  node [left] { } (m-2-1)
(m-1-1.east|-m-1-2) edge  node [above] { } (m-1-2)
(m-2-1.east|-m-2-2) edge  node [below] {} node [below] { } (m-2-2)
(m-1-2) edge  node [right] { } (m-2-2);
\end{tikzpicture}
\end{equation}
Hence, a direct computation yields
\begin{align}
(\mathcal{M}(f) \circ \Phi_V)(g)& = f^*A^g \nonumber\\ 
&=A^g|_U \nonumber\\
&=A^{g|_U} \ \ \ \text{from (\ref{restriction diagram})}\nonumber \\
&=A^{f^*g} \nonumber\\
&=\Phi_U(f^*g) \nonumber\\
&=(\Phi_U \circ \mathcal{E}(f))(g),
\end{align} which gives an ``objectwise" commutativity of the diagram. 
Similarly, for any morphism \begin{equation*}
(g,\varphi):g\longrightarrow g\cdot \varphi=g' \text{ in } Hom_{\mathcal{E}(V)}(g,g'),
\end{equation*}and for each morphism  $f:U \hookrightarrow V$, one has another natural diagram again from the definition and from the restriction functor as above:
\begin{equation} \label{restriction of diagram for morphisms}
\begin{tikzpicture}
\matrix (m) [matrix of math nodes,row sep=2em,column sep=3em,minimum width=2em] {
	\varphi  & \rho_{\varphi} \\
	\varphi|_U & \rho_{\varphi|_U} = \rho_{\varphi}|_U\\};
\path[-stealth]
(m-1-1) edge  node [left] { } (m-2-1)
(m-1-1.east|-m-1-2) edge  node [above] { } (m-1-2)
(m-2-1.east|-m-2-2) edge  node [below] { } (m-2-2)
(m-1-2) edge  node [right] { } (m-2-2);
\end{tikzpicture}
\end{equation}
Therefore, we obtain 
\begin{align}
(\mathcal{M}(f) \circ \Phi_V)(g,\varphi)&= (f^*A^g,f^*\rho_{\varphi}) \nonumber\\ 
&=(A^{f^*g},\rho_{\varphi}|_U) \ \ \ \text{from (\ref{restriction diagram})}\nonumber\\
&=(A^{f^*g},\rho_{\varphi|_U})  \ \ \ \text{from (\ref{restriction of diagram for morphisms})}\nonumber\\
&=(A^{f^*g},\rho_{f^*\varphi}) \nonumber\\
&=\Phi_U(f^*g,f^*\varphi) \nonumber\\
&=(\Phi_U \circ \mathcal{E}(f))(g,\varphi),
\end{align} which implies the desired ``morphismwise" commutativity.

Therefore, $ \Phi $ defines a natural transformation between $\mathcal{E}$ and $\mathcal{M}$ via the collection of natural maps \begin{equation*}
\big\{\Phi_U: \mathcal{E}(U)\longrightarrow \mathcal{M}(U) \big\}_{U\in Ob(\mathcal{C})}
\end{equation*}  such that  for each morphism $f:U\rightarrow V$ in $\mathcal{C}$ the following  diagram commutes. \begin{equation*} 
\begin{tikzpicture}
\matrix (m) [matrix of math nodes,row sep=3em,column sep=4em,minimum width=2em] {
	\mathcal{E}(V)  & \mathcal{M}(V) \\
	\mathcal{E}(U) & \mathcal{M}(V)\\};
\path[-stealth]
(m-1-1) edge  node [left] {{\small $ \mathcal{E}(f) $}} (m-2-1)
edge  node [above] {{\small $ \Phi_V $}} (m-1-2)
(m-2-1.east|-m-2-2) edge  node [below] {} node [below] {{\small $ \Phi_U $}} (m-2-2)
(m-1-2) edge  node [right] {{\small $\mathcal{M}(f)$}} (m-2-2);
\end{tikzpicture}
\end{equation*}

The inverse construction, on the other hand, essentially uses Mess' result \cite{Mess}: For each object $ U $ in $\mathcal{C}$, the map $\Phi_U$ is indeed invertible and the inverse map
\begin{equation*}
\Phi_U^{-1}: \mathcal{M}(U) \rightarrow \mathcal{E}(U)
\end{equation*} is defined as follows: 

Let us first explain the role of   \cite{Mess}. Once we choose a hyperbolic structure on a closed orientable surface $\Sigma$ of genus $g>1$ and view it as a Riemannian surface, then  a flat connection $A$  defines the holonomy representation $H$ of this hyperbolic structure, and hence a Fuchsian representation \cite{Bradlow}. Thus, by Mess' theorem in \cite{Mess}, there exists a suitable flat pseudo-Riemannian manifold $M=\Sigma \times (0,\infty)$ whose flat structure given by a flat pseudo-Riemannian metric, say $g_A$. Moreover, its surface group representation agrees with $H$. Therefore, we have a well-defined assignment on objects \begin{equation*}
\Phi_U^{-1}: \mathcal{M}(U) \rightarrow \mathcal{E}(U), \ \ A\mapsto \ g_A
\end{equation*} such that  the corresponding flat connection $A^{g_A}$ associated with $g_A$ is exactly the connection we started with. This is because the surface group representations agree. That is, we have
\begin{equation*}
\Phi_U \circ \Phi_U^{-1}: A \longmapsto g_A \longmapsto A^{g_A} = A.
\end{equation*}

Likewise, one obtains $ \Phi_U^{-1} \circ \Phi_U |_{\mathcal{E}(U)} = id_{\mathcal{E}(U)}.$ Then, by using a similar analysis as above, it is rather straightforward to check that we have a well-defined assignment $ \Phi^{-1}_U $ on both objects and morphisms, together with an appropriate commutative diagram analogous to the one in Diagram \ref{desired commutative diagram}. Therefore, $ \Phi_U^{-1}$ is a functor of groupoids as well. 

By construction, $ \Phi^{-1} $ is indeed the natural transformation that serves as the inverse of  $\Phi$. This completes the proof of Theorem \ref{THM-2}.
\begin{flushright}
	$ \square $
\end{flushright}
\epf

\end{document}